\documentclass[a4paper,11pt]{amsart}

\usepackage{amsmath, amssymb, amsthm, graphicx, float, textcomp, enumerate, hyperref}
\usepackage[margin=1.2 in]{geometry}
\usepackage[percent]{overpic}
\usepackage[dvipsnames]{xcolor}
 \usepackage[all]{xy}
 
 \newcommand{\co}{\colon\thinspace}

\newtheorem{theor}{Theorem}[section]
\newtheorem{cor}[theor]{Corollary}
\newtheorem{lemma}[theor]{Lemma}
\newtheorem{prop}[theor]{Proposition}

 \newtheorem{theorintro}{Theorem}

    \newtheoremstyle{TheoremNum}
        {\topsep}{\topsep}              
        {\itshape}                      
        {}                              
        {\bfseries}                     
        {.}                             
        { }                             
        {\thmname{#1}\thmnote{ \bfseries #3}}
    \theoremstyle{TheoremNum}

\newtheorem{duplicate}{Theorem}
\theoremstyle{definition}
\newtheorem{defin}[theor]{Definition}
\newtheorem{rmk}[theor]{Remark}

\DeclareMathOperator{\mcg}{MCG}
\DeclareMathOperator{\Homeo}{Homeo}

\DeclareMathOperator{\ai}{\hat{\iota}}
\DeclareMathOperator{\R}{\mathbb{R}}
\DeclareMathOperator{\Z}{\mathbb{Z}}
\DeclareMathOperator{\N}{\mathbb{N}}
\DeclareMathOperator{\Sp}{Sp}

\DeclareMathOperator{\Aut}{Aut}
\DeclareMathOperator{\Ends}{Ends}
\DeclareMathOperator{\H1}{H_1}
\DeclareMathOperator{\C1}{H^1}

\DeclareMathOperator{\filt}{\mathcal{F}}
\DeclareMathOperator{\supp}{supp}
\DeclareMathOperator{\lends}{\mathcal{L}}
\DeclareMathOperator{\rends}{\mathcal{R}}
\DeclareMathOperator{\ssm}{\smallsetminus}

\newcommand{\st}{\,|\,}

\DeclareMathOperator{\hs}{\mathcal{FS}}

\DeclareMathOperator{\Span}{\rm{Span}}

\title{Big mapping class groups acting on homology}
\author{Federica Fanoni}
\address{Institut de Recherche Math\'ematique Avanc\'ee\\UMR 7501, Universit\'e de Strasbourg et CNRS\\7 rue Ren\'e Descartes\\67000 Strasbourg, France}
\email{federica.fanoni@gmail.com}
\author{Sebastian Hensel}
\address{Mathematisches Institut der Universit\"at M\"unchen\\Theresienstr. 39\\D-80333 M\"unchen, Germany}
\email{hensel@math.lmu.de}
\author{Nicholas G.~Vlamis}
\address{CUNY Queens College, 65-30 Kissena Blvd, Flushing, NY 11367}
\email{nicholas.vlamis@qc.cuny.edu}

\date{\today}

\begin{document}
\maketitle

\begin{abstract} 
We study the action of (big) mapping class groups on the first homology of the corresponding surface. We give a precise characterization of the image of the induced homology representation.
\end{abstract}

\section{Introduction}

Surfaces are among the most basic and most fundamental objects in
geometry and topology. Although, as spaces, they may seem simple to
understand, their symmetries -- mapping classes -- certainly are not.

Given a surface $S$, a first approach to understand its mapping class group $\mcg(S)$ is to consider the natural action on the first homology $\H1(S;\Z)$.
This leads to the \emph{homology representation}
\[\rho_S\co \mcg(S) \to \Aut(\H1(S;\Z)). \]
For a surface $S$ of finite genus $g\geq 1$ with at most one puncture, it is well known that the elements in the image of $\rho_S$ are precisely those which preserve the algebraic intersection form $\ai$ (first shown by \cite{Burkhardt_Grundzuege}, see \cite[Chapter 6]{FM_Primer} for a discussion of the result). 
Usually, this is phrased
as saying that $\rho_S\co\mcg(S)\to\Sp(2g;\Z)$ is surjective (as $\ai$ is a symplectic pairing for such a surface).

\smallskip In this article, we determine the image of $\rho_S$ for any surface and,
in particular, those of infinite type. The first case is that of the Loch Ness monster surface (i.e.\ the surface of infinite genus and one end). Here, the
situation is very similar to the closed case:

\begin{theorintro}\label{monster}
Let $S$ be the Loch Ness monster surface. The image of $\rho_S$ is the group of automorphisms of $\H1(S;\Z)$ that preserve the algebraic intersection form.
\end{theorintro}

As the Loch Ness monster surface is one-ended, \( \ai \) is symplectic and Theorem~\ref{monster} is equivalent to saying that the natural homomorphism \(  \mcg(S) \to \Sp(\N;\Z) \) is surjective (see Section~\ref{sec:ness} for more details).

\smallskip For more general surfaces, the situation is more
complicated.  For finite-type surfaces, the mapping class group
permutes the punctures (and therefore the homology classes they
define). For an infinite-type surface, one similarly has to encode the
structure of the ends of $S$ in homology to capture the action of the
mapping class group on ends. We do this by defining the \emph{homology
  end filtration} $\filt$ of $\H1(S;\Z)$.
It consists of the collection of the homologies of
unbounded subsurfaces with a single boundary component\footnote{with an extra technical condition -- see Definitions \ref{defin:subs} and \ref{defin:filt}.}. Further, 
for a homology class $[\delta]$ defined by a separating, oriented, simple,
closed curve, we denote by $\lends([\delta])$ the set of ends of $S$ to
the left of $\delta$ (this is well-defined by Lemma~\ref{lem:twosep}).

With this terminology, we can state our main result as follows.
\begin{theorintro}\label{mainthm}
  Let $S$ be an infinite-type surface, different from the Loch Ness monster and the once-punctured Loch Ness monster.  If
  $\phi$ is an automorphism of $\H1(S;\Z)$ preserving both $\ai$ and $\filt$,
  then the following hold:
  \begin{enumerate}[i)]
  \item Exactly one of $\phi$ and $-\phi$ lies in the image of $\rho_S$.
  \item $\phi$ preserves homology classes defined by separating simple
    closed curves.
  \item $\phi$ determines a homeomorphism $f_\phi$ of the space of ends of $S$, 
    and $\phi$ lies in the image of $\rho_S$ exactly if
    \[ f_\phi(\lends([\delta])) = \lends(\phi([\delta])) \]
    for some (hence any) simple separating closed curve $\delta$ which is non-trivial in 
    $H_1(S;\Z)$.
  \end{enumerate}
\end{theorintro}
Actually, we can show that the theorem holds for finite-type surfaces with at least four punctures. Furthermore, we can also characterize the image of $\rho_S$ in the case of the once-punctured Loch Ness monster (see Section \ref{sec:ness}).

\smallskip We emphasize that even the proof of Theorem~\ref{monster} already requires ideas not necessary in the finite-type case. Namely, in the classical case one starts with a collection of simple closed curves $\alpha_i, \beta_i$ intersecting in a standard pattern and realizes the classes $\phi([\alpha_i]), \phi([\beta_i])$ with the correct intersection pattern; it is then easy to construct a mapping class with the correct action. In the Loch Ness monster case, to follow this approach one also needs to realize the classes $\phi([\alpha_i]), \phi([\beta_i])$ by curves not accumulating in any compact subset of $S$. To take care of this, we adapt an argument of Richards \cite{Richards_Classification}; the details are discussed in Section \ref{sec:ness}. 

 To prove Theorem~\ref{mainthm}, the first step is to show that (under the given assumption on the surface) ultrafilters of $\filt$ are in correspondence with the ends of the surface (Lemma \ref{lem:ultrafilters}).
It follows that an automorphism $\phi$ preserving $\filt$ induces a permutation of its ultrafilters and hence a map $f_\phi$ of $\Ends(S)$ (Proposition \ref{prop:end-homeo}).

The second step is to deal with homology classes of separating simple closed curves. We note that two such curves induce the same class in homology if and only if the set of ends to the left of one is the same as the set of ends to the left of the other (Lemma \ref{lem:twosep}). Furthermore, we can show that these classes can be detected using $\filt$ (Proposition \ref{prop:simple-preserving}), which implies that they are permuted by any $\phi$ satisfying the hypotheses of the theorem.

To finish the proof, we use again a variation of the same argument of Richards that we employ for the Loch Ness monster case. While the structure of this step is the same in both cases, having to deal with more ends renders the proof less transparent.

\smallskip A natural complement of our study is the investigation of
the kernel of $\rho_S$, called the \emph{Torelli group} of $S$. For
finite-type surface this has been the subject of a sizeable amount of
research (the survey \cite{JohnsonSurvey} gives an
excellent overview over the by-now classical theory). In recent years,
more progress has been made, and the Torelli group is by now fairly
well understood.

Recently, the Torelli group has been investigated
for infinite-type surfaces as well by Aramayona, Ghaswala, Kent,
McLeay, Tao and Winarski \cite{AramayonaBig}. Among the results they
obtain, they characterize which elements belong to this subgroup by
showing that the Torelli group of an infinite-type surface is
topologically generated by its compactly-supported elements and hence
by separating Dehn twists and bounding pair maps.

\subsection{Necessity of the conditions}
In this section we will discuss how all the conditions in Theorem \ref{mainthm} are necessary, by providing examples of automorphisms not induced by mapping classes where one of the hypotheses is not satisfied.

\smallskip Already finite-type surfaces with punctures show that
preserving the algebraic intersection pairing is not sufficient to
guarantee realizability. Indeed, mapping classes of the closed genus-\( g \) surface with \( n \) punctures permute
the punctures, and therefore the mapping class group acts on the
isotropic subspace as a permutation representation
and this fact is not seen by \( \ai \).

\smallskip More interesting examples can be constructed on Jacob's ladder surface, i.e.~the two-ended infinite-genus surface with no planar ends. We consider the homology basis given by the curves depicted in Figure \ref{fig:basisJacob}.
\begin{figure}[t]
	\vspace*{.5cm}
	\begin{overpic}{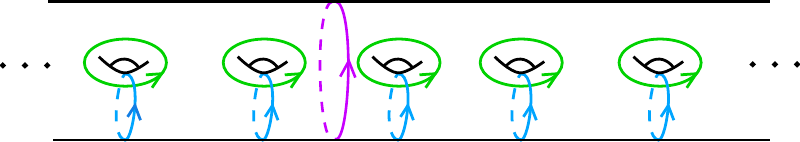}
		\put(12,14){$\alpha_{-2}$}
		\put(30,14){$\alpha_{-1}$}
		\put(48,14){$\alpha_{0}$}
		\put(64,14){$\alpha_{1}$}
		\put(80,14){$\alpha_{2}$}
		\put(41,21){$\gamma$}
		\put(12,-6){$\beta_{-2}$}
		\put(30,-6){$\beta_{-1}$}
		\put(48,-6){$\beta_{0}$}
		\put(64,-6){$\beta_{1}$}
		\put(80,-6){$\beta_{2}$}
	\end{overpic}
	\vspace*{.5cm}
	\caption{Curves for a homology basis of Jacob's ladder}\label{fig:basisJacob}
\end{figure}
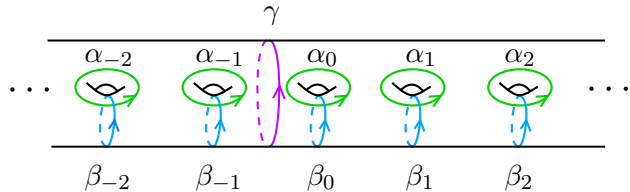

Consider the automorphism $\phi_1$ fixing $[\gamma]$ and $[\alpha_i],[\beta_i]$ for $i$ even and sending $[\alpha_i],[\beta_i]$ to $[\alpha_{-i}],[\beta_{-i}]$, respectively, for $i$ odd.

The automorphism \( \phi_1 \) cannot be realized by a mapping class because 
the sequence of curves \( \{\alpha_n\}_{n\in \N} \) exit one end, but the representatives of the images under \( \phi_1 \) accumulate to both ends.
One can check that that $\phi_1$ does not preserve the homology end filtration: more precisely, it can be proved that if we denote by $X$ the subsurface to the left of $\gamma$, then $\phi_1(\H1(X;\Z))$ is not in $\filt$.
This example shows how the homology end filtration is important to control which ends are accumulated by non-isotropic vectors.

Next, consider the automorphism $\phi_2$ which fixes all basis elements except for $[\gamma]$, which is mapped to $-[\gamma]$. 
Again, $\phi_2$ cannot be induced by a mapping class. This time it is because the sequence of curves \( \{\alpha_n\}_{n\in\N} \) exit the end to the right of $\gamma$, but to the left of any representative of $-[\gamma]$ (e.g.~ $\gamma$ with the opposite orientation).  In terms of the condition in Lemma \ref{lem:single-coherence},
$$f_{\phi_2}(\lends([\gamma]))\neq \lends (\phi_2([\gamma])).$$
On the other hand, $-\phi_2$ is induced by a mapping class (the involution which can be informally described as the rotation of angle $\pi$ around an axis joining the two ends of $S$, see Figure \ref{fig:rotation}).
\begin{figure}[ht]
\includegraphics{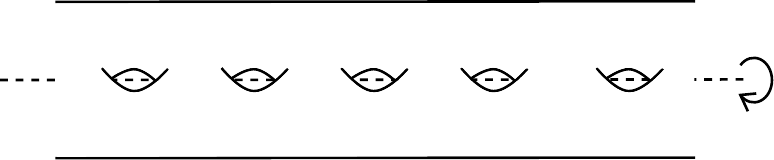}
\caption{A mapping class inducing $-\phi_2$}\label{fig:rotation}
\end{figure}

Note that it is also not enough to require that algebraic intersection and topological type of curves be preserved (where by this we mean that the image of the class of a simple closed curve is the class of a simple closed curve in the same mapping class group orbit), as shown by $\phi_1$ and $\phi_2$.

Finally, one could wonder if there is a characterization of the image of $\rho_S$ in terms of the set of simple isotropic classes instead of the homology end filtration; we comment on this in Section \ref{sec:charwithis}.

\subsection{Structure of the paper}
After some preliminaries about surfaces and their homology (Section \ref{sec:prelim}), we deal with the case of the Loch Ness monster with at most one puncture in Section \ref{sec:ness}. The proof of Theorem \ref{monster} contains many of the ideas that are necessary for the general case, but it is simpler since there is only one end.

In section \ref{sec:endfiltration} we introduce the main new tool, the homology end filtration and we prove the main result (Theorem \ref{mainthm}) in Section \ref{sec:proof}.

We end the paper with an appendix collecting some realization results for homology and cohomology classes: characterizations of homology classes represented by simple closed curves that the authors could not find in the literature (which may be of independent interest) and a description of which cohomology classes are given by intersection with proper arcs joining two ends.

\section*{Acknowledgements}
The first named author would like to thank Anna Wienhard for asking the question that started the project and Gabriele Viaggi for useful discussions. She was partially supported by the project \emph{Rigidity, deformations and limits of maximal representations} of the DFG Priority Program SPP 2026 \emph{Geometry at infinity}.

The third named author was supported in part by NSF RTG grant 1045119.

\section{Preliminaries}\label{sec:prelim}
Throughout, a \emph{surface} will refer to an oriented, connected, second countable, Hausdorff two-dimensional manifold. 
Unless stated otherwise, a surface does not have boundary -- the one notable
exception being subsurfaces of other surfaces.
A surface is \emph{of finite type} if its fundamental group is finitely generated and \emph{of infinite type} otherwise.

The \emph{mapping class group} of the surface $S$ is the group of orientation preserving homeomorphisms of $S$ up to isotopy:
\[\mcg(S) = \mathrm{Homeo}^+(S)/\mathrm{isotopy}.\]

Throughout the article, a \emph{curve} will refer to a simple, closed, oriented curve; in addition, we will routinely conflate the isotopy class of a curve with a representative.

A curve is \emph{essential} if it bounds neither a disk nor a once-punctured disk; it is \emph{separating} if its complement is disconnected and \emph{non-separating} otherwise.

When discussing subsurfaces, we assume that every boundary component is an essential curve with the induced orientation, i.e.~such that the subsurface is to the left of the curve.

We will denote by $i$ the \emph{geometric intersection number} of two curves (note: geometric intersection does not take into account the orientation of the curves).
A collection of curves \( \{\alpha_i, \beta_i\}_{i\in I} \) has the \emph{standard (symplectic) intersection pattern} if \( i(\alpha_i, \alpha_j)=0, i(\beta_i, \beta_j) = 0 \), and \( i(\alpha_i, \beta_j) = \delta_{ij} \) for all \( i,j \in I \). 

An \emph{arc} in a surface is the image of a proper embedding of either \( (0,1) \), \([0,1)\) or \([0,1]\) into the surface. When a boundary point of the interval is included, the corresponding point on the surface must belong to a boundary component.
As with curves, we do not distinguish between an arc and its isotopy class (isotopies of arcs are taken relative to the boundary where appropriate).

\subsection{Ends of a surface}
An \emph{end} of a surface is an equivalence class of a descending chain $U_1\supset U_2\supset \dots$ of open connected subsurfaces with compact boundary and such that for any compact $K$ there is an index $n_K$ such that for all $n\geq n_K$, $K\cap U_n=\emptyset$.
Two such chains $U_1\supset U_2\supset \dots$ and $V_1\supset V_2\supset \dots$ are equivalent if for every $n$ there is an $N$ such that $U_N\subset V_n$ and $V_N\subset U_n$.

The \emph{space of ends} $\Ends(S)$ is the set of ends endowed with the topology generated by sets of the form $U^*$, where $U$ is an open subset with compact boundary, and
$$U^*:=\{[U_1\supset U_2\supset\dots]\st \exists n : U_n\subset U\}.$$

An end \( [U_1 \supset U_2 \supset\dots] \) is \emph{planar} if there exists an integer \( n \) such that \( U_n \) is homeomorphic to a subset of the plane (or, equivalently, has genus $0$).
Otherwise, the end is \emph{non-planar}, and every \( U_n \) has infinite genus\footnote{Sometimes in the literature a non-planar end is also referred to as \emph{accumulated by genus} as every neighborhood has infinite genus.}. 
An end is \emph{isolated} if it is an isolated point of the space of ends.
We will routinely refer to an isolated planar end as a \emph{puncture}.

It is easy to check that $\Ends_g(S)$, the subset of non-planar ends, is a closed subset of $\Ends(S)$.

Ker\'ekj\'art\'o and Richards \cite{Richards_Classification} showed that surfaces are topologically classified by the triple $(g, (\Ends(S),\Ends_g(S)))$, where $g\in\N\cup\{0,\infty\}$ is the genus and $(\Ends(S),\Ends_g(S))$ is considered as a pair of topological spaces, up to homeomorphism.

\subsection{Homology of surfaces}
The main focus of this article is the first homology of a surface considered with integral coefficients; accordingly, when referring to the homology of a surface \( S \), we are referring to \( \H1(S;\Z) \).

Every homology class in \( \H1(S;\Z) \) can be represented by a -- possibly non-simple -- loop in \( S \). 
Given a homology class \( x \in \H1(S;\Z) \), we say that \( x \) is \emph{simple} if there is a simple closed curve \( \alpha \) such \( [\alpha]= x \).
In this case, we say that \( x \) is \emph{represented by \( \alpha \)}.

The \emph{algebraic intersection number}, denoted $\ai$, defines a bilinear, antisymmetric form on $\H1(S;\Z)$.
An element \( x \)  of \( \H1(S;\Z) \) is \emph{isotropic} if \( \ai(x,y) = 0 \) for every \( y \in \H1(S;\Z) \).
If neither complementary component of a separating curve on a non-compact surface has compact closure, then the curve is non-trivial in homology; hence, we see that the form $\ai$ is symplectic if and only if $|\Ends(S)|\leq 1$.
Note that  if \( x \) is a simple (non-)isotropic homology class and \( \alpha \) is a curve representing \( x \), then \( \alpha \) is (non-)separating.

Also note that if $a$ is an arc, algebraic intersection of homology
classes with $a$ is a well defined linear functional
$\ai(a, \cdot):\H1(S;\mathbb{Z})\to\mathbb{Z}$ and hence gives a cohomology class in \( H^1(S;\Z) \).

Throughout the paper we will be interested into two special types of subsurfaces, \emph{star} and \emph{flare} surfaces.
\begin{defin}\label{defin:subs}
A \emph{star surface} is a connected finite-type subsurface so that all boundary components are separating curves in $S$ and all complementary components are unbounded.\\
A \emph{flare surface} is an unbounded subsurface $X$ with a single boundary component, which is separating, and such that the closure of $S\ssm X$ is not a finite-type surface with at most one puncture.
\end{defin}
For a star or flare surface $X$, we will denote by $\H1(X;\Z)$ the image of the homology of $X$ under the monomorphism induced by the inclusion $X\hookrightarrow S$. Note that for a general subsurface the map in homology need not be injective as the image of some boundary components might be zero.

\subsection{Homology classes of simple closed curves}
In our study, we require some results on the interaction of simple closed curves with homology classes. 
The first lemma is a criterion to detect simple non-isotropics. The proof is standard (also in the infinite-type setting)
and is delegated to Appendix~\ref{sec:curves}. As mentioned in the introduction, in the appendix we also collect a number of further results on the interplay between homology and simple representability that are not required for the main argument, but may be of independent interest.
\begin{lemma}\label{lem:detectingnonsep1}
	Let $S$ be any surface and $x \in \H1(S;\Z)$.
	Then $x$ is a simple non-isotropic if and only if there exists $y \in \H1(S; \Z)$ such that $\ai(x,y) = 1$. \qed
\end{lemma}
The complementary components of a separating curve \( \gamma \) in a surface \( S \) determine two disjoint clopen sets \( \lends(\gamma) \) and \( \rends(\gamma) \) that partition \( \Ends(S) \).
The sets are labelled so that, when considering the orientation of \( \gamma \), the set \( \lends(\gamma) \) consists of the ends to the left of \( \gamma \) and \( \rends(\gamma) \) those to the right.

First, we give a lemma determining when simple separating curves define the same homology class.
\begin{lemma}\label{lem:twosep}
	Let $S$ be any surface and let $\alpha, \beta$ be two separating simple closed curves. Then
	$[\alpha] = [\beta]$ if and only if $\lends(\alpha) = \lends(\beta)$.
\end{lemma}
\begin{proof}
	First observe that two ends are on different sides of $\alpha$ exactly if there is an arc connecting these
	ends so that $\ai([\alpha], a) \neq 0$. This shows that homologous curves induce the same decomposition of
	ends.
	
Now, suppose the set of ends to the left of $\alpha$ is the same as the set of ends to the left of $\beta$. 	Let $\Sigma$ be a compact subsurface containing $\alpha\cup\beta$ (where we allow $\Sigma$ to have boundary components homotopic to punctures), such that all connected components  $\{X_i\st i\in I\}$ of $S\smallsetminus \Sigma$ are unbounded.
	Since $\alpha$ and $\beta$ induce the same partition of ends, a surface $X_i$ is to the right of $\alpha$ if and only if it is to the right of $\beta$. But then if $I_r$ is the set of indices $i$ such that $X_i$ is to the right of $\alpha$, we have
	$$[\alpha]=\sum_{i\in I_r}\sum_{\gamma\subset\partial X_i}[\gamma]=[\beta].$$
\end{proof}

\begin{lemma}\label{lem:outside-curves}
	Let $S$ be any surface and let $X\subset S$ be a subsurface with separating boundary components.
	Suppose that $\alpha$ is a simple closed curve which is disjoint from $X$.
	If $[\alpha]=[\beta]$ for some loop $\beta\subset X$ , then $[\alpha] = \pm[\partial_j X]$, where $\partial_j X$ is one of the boundary curves of $X$.
\end{lemma}
\begin{proof}
Let $i\co X\hookrightarrow S$ be the inclusion and $i_*$ the map induced on homology.
	As \( \alpha \) is disjoint from \( X \), we have that \( \ai([\alpha],v) = 0 \) for all \( v \in i_*\H1(X,\Z) \) and, as \( [\alpha] \in i_*\H1(X, \Z) \), we can conclude that \( \ai([\alpha], v) = 0 \) for all \( v \in \H1(S, \Z) \). 
	So, \( [\alpha] \) is isotropic and hence \( \alpha \) is separating. 

	Without loss of generality, suppose that \( X \) is to the left of \( \alpha \).
	Now, there exists a unique \( j \) such that \( \alpha \) is to the right of \( \partial_j X \), which implies  \( \lends(\partial_j X) \subseteq \lends(\alpha) \).
	If equality holds, then \( [\alpha] = [\partial_j X] \).
	If equality fails, then there is an arc \( a \) in \( S \ssm X \) connecting an end in \( \lends(\alpha) \ssm \lends(\partial X) \) to an end in \( \rends(\alpha) \).
	It follows that \( |\ai([\alpha], a)| = 1 \) and that  \( \ai(x,a)=0 \) for every \( x \in i_*\H1(X;\Z) \), which is a contradiction.
\end{proof}

\section{The Loch Ness Monster surface}\label{sec:ness}
In this section we discuss the case of the Loch Ness monster surface -- the infinite-genus surface with a unique end -- and of the once-punctured Loch Ness Monster. 
For these surfaces, the complication of preserving the structure of ends is not necessary, and so the result
takes a form very reminiscent of the closed case.

\begin{theor}\label{thm:monsters}
	If $S$ is the Loch Ness monster surface with at most one puncture, then the map $\rho_S:\mcg(S)\to\Aut(\H1(S;\Z))$	is a surjection onto the group of automorphisms of homology preserving the algebraic intersection form and acting as the identity on the isotropic subspace.
\end{theor}

Note that for the Loch Ness monster the isotropic subspace is trivial, so we recover Theorem \ref{monster}.

\begin{proof}
We first prove the result for the Loch Ness monster $L$, where the condition on the action on the isotropic subspace is void.

	Clearly, a mapping class preserves algebraic intersection, so we just need to prove that if $\phi$ is an automorphism preserving $\ai$, then it is induced by a mapping class.
	
	Fix a compact exhaustion $\{\Sigma_n\}_{n\in\N}$ of $L$, where $\Sigma_n$ has genus $n$ and connected boundary. We want to construct two sequences of subsurfaces $\{A_n\}$ and $\{B_n\}$, each with connected boundary, and homeomorphisms $f_n\co A_n\to B_n$ such that:
	\begin{enumerate}
		\item $\Sigma_n\subset A_n$ for every odd $n$ and $\Sigma_n\subset B_n$ for every even $n$,
		\item \( f_n |_{A_{n-1}} = f_{n-1} \), and
		\item the induced homomorphism \( (f_n)_* \co H_1(A_n ; \Z) \to H_1(B_n; \Z) \) agrees with $\phi\big|_{\H1(A_n;\Z)}$.
	\end{enumerate}
	Note that condition~(1) implies that both sequences \( \{A_n\}
        \) and \( \{B_n\} \) are exhaustions. Therefore, using
        condition~(2) implies that we can take the direct limit\footnote{Formally, we view \( \{A_n\} \) and \( \{B_n\} \) as directed systems with respect to inclusion and,  as both sequences are exhaustions of \( S \), both of their direct limits are exactly \( S \). It is in this setting that we use the universal property of direct limits to obtain the map \( f \).}
        of the \( f_n \), and the resulting map $f$ is a homeomorphism of $L$. 
        Condition~(3) then implies that $f$ acts as $\phi$ on homology.
	
	We construct the desired sequence of subsurfaces via induction.
	
	\textbf{Base case:} Set $g_1=1$ and $A_1=\Sigma_1$.
	Choose a geometric homology basis $\alpha_1,\beta_1$ of $\H1(\Sigma_1;\Z)$ and realize the image classes $\phi([\alpha_1]),\phi([\beta_1])$ by non-separating curves $\alpha_1',\beta_1'$ intersecting once  (as in \cite[Theorem 6.4]{FM_Primer}).
	Let $B_1$ be the one-holed torus obtained by taking a regular neighborhood of \( \alpha_1'\cup \beta_1' \) and let $f_1$ be a homeomorphism between $A_1$ and $B_1$ sending $\alpha_1$ to $\alpha_1'$ and $\beta_1$ to $\beta_1'$.
	
	\textbf{Induction step:} Suppose that we are given $A_n$, $B_n$ and $f_n$ satisfying conditions (1)-(3) above.
	
	If $n$ is even, set $A_{n+1}:=\Sigma_m$, where $m\geq n+1$ is such that $A_n\subset\Sigma_m$ and $A_n$ is not homotopic to $\Sigma_m$.
	Set $g_{n+1}$ to be the genus of $A_{n+1}$.
	Choose curves $\alpha_{g_n+1},\dots,\beta_{g_{n+1}}$ in $A_{n+1}\ssm A_n$ with the standard intersection pattern.
	Note that the images $\phi([\alpha_i])$, $\phi([\beta_i])$, for $i>g_n$, belong to $\H1(S\ssm B_n;\Z)$, as they have algebraic intersection zero with all vectors in a basis for $\H1(B_n;\Z)$.
	Hence, as the $B_n$ have a single boundary component, the classes can be realized by curves $\alpha_i',\beta_i'$ outside $B_n$ and with the standard intersection pattern.
	Let $B_{n+1}$ be a genus $g_{n+1}$ surface with one boundary component containing $B_n$ and all the curves constructed.
	$B_{n+1}\ssm B_n$ and $A_{n+1}\ssm A_n$ are both surfaces with genus $g_{n+1}-g_n$ with two boundary components, so we can extend $f_n$ to a homeomorphism $f_{n+1}$ sending the $\alpha_i,\beta_i$, for $g_n+1\leq i\leq g_{n+1}$ to the corresponding $\alpha_i',\beta_i'.$
	
	If $n$ is odd, the argument is similar: set \( B_{n+1} = \Sigma_m \), where \( m \geq n+1 \) is such that \( B_n \subset \Sigma_m \) is not homotopic to \( \Sigma_m \).
	Proceed identically to the above case switching the roles of \( A_n, \alpha_i \), and \( \beta_i \) with those of \( B_n, \alpha_i',  \) and \( \beta_i' \), respectively, in every instance.
	Now after constructing \( A_{n+1} \), we extend \( f_n^{-1} \co B_n \to A_n \) to a homeomorphism \( h \co B_{n+1} \to A_{n+1} \) mapping \( \alpha_i', \beta_i' \), for \( g_{n}+1 \leq i \leq g_{n+1} \), to the corresponding \( \alpha_i, \beta_i \).
	We finish by setting \( f_{n+1} = h^{-1} \).

In the case of the once-punctured Loch Ness monster surface \( L' \), we take an exhaustion of finite-type surface  $\{\Sigma_n\}_{n\in\N}$  so that \( \Sigma_n \) has genus $n$, connected boundary, and contains the unique puncture of \( L' \). The same proof then yields the result.
\end{proof}

\begin{rmk}
The role of alternating between constructing \( A_n \) and \( B_n \) is a bit subtle: the main purpose is that doing so allows us to simultaneously build both \( f \) and \( f^{-1} \).
If we only constructed the \( A_n \), we would not be able to guarantee that the resulting map \( f \) is a homeomorphism: the issue is that there are non-surjective embeddings of infinite-type surfaces into themselves and such maps can arise as direct limits.
In this case, the boundary curves of the images of the \( A_n \) would have to accumulate in \( S \). 
Therefore, one should view this alternating technique as a means to avoid this accumulation issue.
Note that this technique appears in Richards's paper on the classification of surfaces \cite{Richards_Classification}.
\end{rmk}

\subsection{The infinite-degree integral symplectic group.}

Consider an infinite-rank \( \Z \)-module $V$ with a countable basis $\{a_i,b_i\st i\in \N\}$ and a symplectic form $\omega$ such that for every $i,j\in \N$
\begin{align*}
\omega(a_i,b_j)&=\delta_{i,j}\\
\omega(a_i,a_j)&=\omega(b_i,b_j)=0.
\end{align*}

The \emph{infinite-degree integral symplectic group} $\Sp(\N;\Z)$ is the group of linear automorphisms of $V$ preserving $\omega$. 
It is clear that the group of automorphisms of \( \H1(L ; \Z) \) preserving \( \ai \) is isomorphic to \( \Sp(\N; \Z) \). 
Under this isomorphism, we have the immediate corollary of Theorem \ref{monster}:

\begin{cor}
If \( L \) is the Loch Ness monster surface, then the action of \( \mcg(L) \) on \( H_1(S; \Z) \) induces an epimorphism \( \mcg(L) \to \Sp(\N; \Z) \).  \qed
\end{cor}

\begin{rmk}
We get an epimorphism $\mcg(S)\to \Sp(\N;\Z)$ also if $S$ is the once-punctured Loch Ness monster, by looking at the action on the quotient of homology by its isotropic subspace.
\end{rmk}

We endow $\Sp(\N;\Z)$  with the topology whose subbasis is given by sets of the form
\[
U_v = \{ A \in \Sp(\N, \R) \st Av = v \}
\]
and their left translates.
This topology, often referred to as the \emph{permutation topology}, turns $\Sp(\N;\Z)$ into a topological group.
We also consider \( \mcg(L) \) as topological group by endowing it with the quotient topology coming from \( \Homeo^+(S) \) equipped with the compact-open topology.
Using the curve graph, this topology on \( \mcg(L) \) can also be described as a permutation topology (see \cite[Section 2.4]{APV_Cohomology} for details).
In particular, one can readily show that the homomorphism \( \mcg(L) \to \Sp(\N; \Z) \) is continuous.

For any $g$, we can naturally embed $\Sp(2g;\Z)$ in $\Sp(\N;\Z)$; this is accomplished by making any element of $\Sp(2g;\Z)$ act on the first $2g$ basis vectors and extending it to the identity on the other basis vectors.
Similarly, we have natural inclusions of $\Sp(2g;\Z)$ in $\Sp(2g';\Z)$ for every $g\leq g'$.
This gives us a directed system and we can consider the direct limit $\Sp(2\infty; \Z) = \varinjlim \Sp(2g;\Z)$, which is a proper subgroup of $\Sp(\N;\Z)$.

The obvious analogy is to consider the directed system of mapping class groups of surfaces $S_{g,1}$ of genus $g$ with one boundary component.
A mapping class is \emph{compactly supported} if it can be represented by a homeomorphism which is the identity outside of a compact set.
The direct limit $\varinjlim\mcg(S_{g,1})$ is the subgroup of compactly supported mapping classes $\mcg_c(L)$ of the Loch Ness monster.

The above discussion yields the following commutative diagram of topological groups:
\[
\xymatrix{\mcg(S_{g,1}) \ar@{^{(}->}[r]\ar@{->>}[d]  & \mcg_c(L)\ar@{^{(}->}[r]\ar@{->>}[d] &\mcg(L)\ar@{->>}[d]  \\
	\Sp(2g;\Z)\ar@{^{(}->}[r] & \Sp(2\infty;\Z)\ar@{^{(}->}[r] &\Sp(\N;\Z)}
\]
where all maps are continuous.
For the Loch Ness monster,  \( \mcg_c(L) \) is dense in \( \mcg(L) \) \cite[Theorem 4]{PV_Algebraic}.
Therefore, the surjectivity of \( \mcg(L) \to \Sp(\N; \Z) \) tells us that \( \Sp(2\infty, \Z) \) is dense in \( \Sp(\N; \Z) \).
(Following the the proof of \cite[Theorem 4]{PV_Algebraic}, one can prove this directly as well.)

\begin{cor}
\( \Sp(2\infty, \Z) \) is dense in \( \Sp(\N;\Z) \). \qed
\end{cor}

\section{The homology end filtration}\label{sec:endfiltration}

In this section we introduce an extra structure associated to the homology of a surface, called  the homology end filtration. Its main purpose is to capture the necessary information of the space of ends of the surface.
This structure is a poset of a class of submodules of $\H1(S;\Z)$ whose space of ultrafilters will correspond to the space of ends of the surface.
This will give us a way to associate a self map of $(\Ends(S),\Ends_g(S))$ to an automorphism of homology preserving the homology end filtration.

Throughout this section we routinely require an additional condition on a surface, which we denote \( (\star) \) and is defined as follows:

\begin{quote}
A surface satisfies \( (\star) \) if it is either planar with at least 4 ends; of finite positive genus with at least 3 ends; or infinite-genus and not homeomorphic to either the Loch Ness monster or the once-punctured Loch Ness monster surface.
\end{quote}

\subsection{Flare surfaces and their homology}
Recall that a \emph{flare surface} is an unbounded subsurface $X$ whose boundary is a single separating simple closed curve and such that $\rends(\partial X)$ is neither empty nor a single puncture.
Note that by definition of flare surface, its boundary is non-trivial in homology.
Let $\hs$ be the set of all flare surfaces.

The main reason why we are interested in these subsurfaces is the following consequence of the definition of the space of ends.

\begin{lemma}\label{lem:clean-basis}
Given a surface $S$ satisfying \( (\star) \), the set
\[ \{ \lends(\partial X) : X\in\hs \} \]
is a  subbasis for $\Ends(S)$ consisting of clopen sets. \qed
\end{lemma}

As a flare surface \( X \) in \( S \) is a closed subset of \( S \), the inclusion \( X \hookrightarrow S \) is a proper map and hence induces a map \( \Ends(X) \to \Ends(S) \); moreover, the fact that the boundary of \( X \) is connected guarantees that this map is injective.
In particular, it is a homeomorphism onto its image, which allows us to naturally identify \( \Ends(X) \) with \( \lends(\partial X) \).

A consequence of Lemma \ref{lem:clean-basis} and of the fact that the ends space is Hausdorff is the following:
\begin{lemma}\label{lem:separate-ends}
Suppose that $Y$ is a flare surface with at least three ends and $e,e'\in \lends(\partial Y)$ are distinct elements.
Then there is a flare surface $X\subset Y$ so that $e\in \lends(\partial X), e'\notin\lends(\partial X)$. \qed
\end{lemma}

We now show that inclusion of homologies of flare surfaces gives inclusion of the corresponding spaces of ends.
\begin{lemma}\label{lem:endset}
If  \( X \) and \( Y \) are two flare surfaces such that
\( \H1(X;\Z) \leq \H1(Y;\Z) \),
then
\( \lends(\partial X) \subseteq \lends (\partial Y). \)
Moreover, if $\H1(X;\Z) = \H1(Y;\Z)$, then $\lends(\partial X)=\lends(\partial Y)$.
\end{lemma}

\begin{proof} For a contradiction, suppose that \( \lends(\partial X) \not\subset \lends(\partial Y) \). 

First suppose that \( \lends(\partial X) \smallsetminus \lends(\partial Y) \) has at least two elements and let \( e_1 \) and \( e_2 \) be two such ends.
Using that \( \partial X \cup \partial Y \) is compact, \( \lends(\partial X) \smallsetminus \lends(\partial Y) \) is clopen, and \( \Ends(S) \) is Hausdorff, there exist simple separating closed curves \( \gamma_1 \) and \( \gamma_2 \) in \( X \smallsetminus Y \) such that \( e_i \in \rends(\gamma_i) \), \( \gamma_1 \cap \gamma_2 = \emptyset \) and $e_i \notin \rends(\gamma_j)$ if $i\neq j$.
As \( [\gamma_i] \in \H1(X, \Z) < \H1(Y, \Z) \), by Lemma \ref{lem:outside-curves} we have that \( [\gamma_i] = \pm[\partial Y] \) for \( i \in \{1,2\} \).
But as $e_i\in \rends(\partial Y)$,  we get $[\gamma_1]=[\partial Y]=[\gamma_2]$, which is impossible since \( \lends(\gamma_1) \neq \lends(\gamma_2) \).

So, we may assume that \( \lends(\partial X) \smallsetminus \lends(\partial Y) \) contains a single end, call it \( e \). 
Repeating the same argument, we can find a simple separating closed curve \( \gamma \) contained in \( X \smallsetminus Y \) such that \( e \in \rends(\gamma) \).
Since \( |\lends(\partial X) \smallsetminus \lends(\partial Y)|=1 \), we have that $\rends(\gamma)=\{e\}$.
Again, we find \( [\gamma] = [\partial Y] \) implying \( \Ends(S) \smallsetminus \lends(\partial Y) = \{e\} \).
By the definition of flare surface,  \( e \) must be non-planar; hence,  \( X\smallsetminus Y \) has infinite genus and \( \H1(X, \Z) \) cannot be a subspace of \( \H1(Y, \Z) \), a contradiction.
\end{proof}

This lemma is the motivation for requiring that $\rends(\partial X)$ not be a single puncture for a flare surface \( X \). 
Indeed, if we allowed this, we could, for instance, construct flare surfaces with the same homology but different spaces of ends, as Figure \ref{fig:samehomology-ends} shows.

\begin{figure}[h]
\vspace*{.5cm}
\begin{overpic}{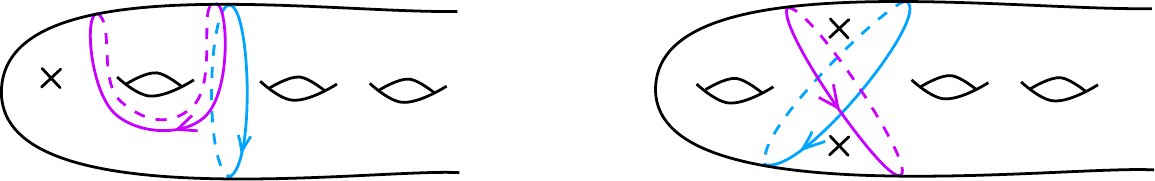}
\put(20,16.5){$\partial X$}
\put(78,16.5){$\partial X$}
\put(6,16){$\partial Y$}
\put(65,16){$\partial Y$}
\end{overpic}
\caption{Two pairs of flare surfaces with the same homology, but different spaces of ends}
\label{fig:samehomology-ends}
\end{figure}
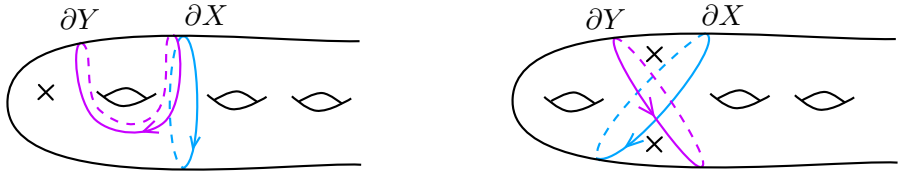

We also note that two disjoint flare surfaces that do not cover the entire space of ends have homologies that intersect trivially:

\begin{lemma}\label{lem:flaredisjoint}
If $X$ and $Y$ are disjoint flare surfaces such that $\lends(\partial X)\cup\lends(\partial Y)\neq \Ends(S)$,
then $\H1(X;\Z)\cap\H1(Y;\Z)=\{0\}.$
Moreover, $\lends(\partial X)\cup\lends(\partial Y)= \Ends(S)$ if and only if $[\partial X]= -[\partial Y]$
\end{lemma}

\begin{proof}
Suppose $x\in\H1(X;\Z)\cap\H1(Y;\Z)$.
Then it must be an isotropic vector: since it can be realized in $X$, it must pair to zero with all vectors that can be realized outside of $X$, but at the same time it can be realized in $Y$ and hence it must pair to zero with all verctors of $\H1(X;\Z)$.
Now we can choose a compact subsurface $K$ with (possibly peripheral) separating boundary components  containing \( \partial X \) and \( \partial Y \) in its interior, and such that \( x \in \H1(K;\Z) \).
Additionally, we choose \( K \) so that there is a single component of \( \partial K \) contained in \( K \ssm (X\cup Y) \).
Let
$$\partial K=\{\gamma_1,\dots,\gamma_n,\delta_1,\dots,\delta_m,\eta\}$$
where $\gamma_1,\dots, \gamma_n$ are curves in $X$,  $\delta_1,\dots, \delta_m$ are curves in $Y$ and $\eta$ is the curve outside $X\cup Y$.
Note that \( \eta \) is homologous to $[\partial X]+[\partial Y]$.

The classes $[\gamma_1],\dots [\gamma_n]$ form a basis for the isotropic subspace of $\H1(K\cap X;\Z)$, the classes $[\delta_1], \dots ,[\delta_m]$ a basis for the isotropic subspace of $\H1(K\cap Y;\Z)$, and all together they form a basis for the isotropic subspace of $\H1(K;\Z)$.
As $x$ must be written at the same time as a linear combination of the $[\gamma_i]$ and as a linear combination of the $[\delta_j]$, it must be the zero vector.

The second part of lemma is now a direct consequence of Lemma \ref{lem:twosep}.
\end{proof}

\begin{figure}[H]
\begin{overpic}{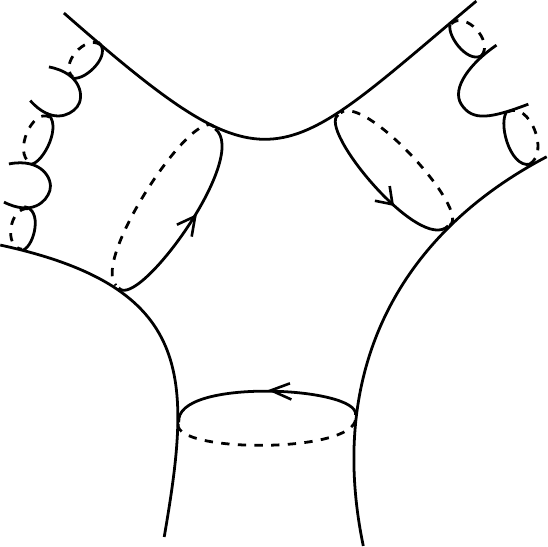}
\put(15,38){$\partial X$}
\put(80,50){$\partial Y$}
\put(47,47){$K$}
\put(-6,57){$\gamma_1$}
\put(-3,75){$\gamma_2$}
\put(2,92){$\gamma_3$}
\put(90,95){$\delta_1$}
\put(100,75){$\delta_2$}
\put(47,32){$\eta$}
\end{overpic}
\caption{Disjoint flare surfaces whose homologies have trivial intersection and the subsurface $K$ in the proof of Lemma \ref{lem:flaredisjoint}.}
\label{fig:flaredisjoint}
\end{figure}
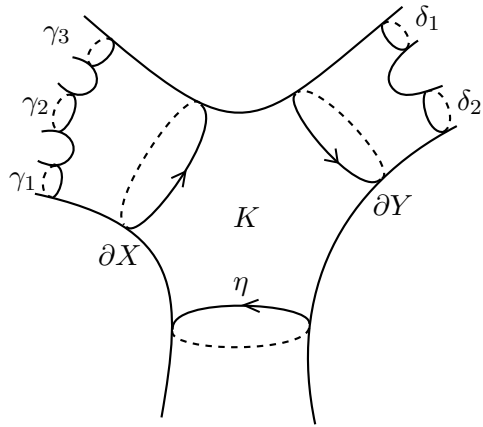

We end this section by showing how nesting at the homology level can
be translated into geometric nesting in the case of flare
surfaces. Since complicated mapping classes can act trivially on the
homology of the surface, geometrically intersecting surfaces can have
nested homologies.
However, we will show that we can find a nested flare 
surface with the correct homology.

We prove first this type of result in the finite-type case.
\begin{lemma}\label{lem:nested-finite-type}
Let $K$ be a finite-type surface and $X', Y \subset K$ two subsurfaces each cut off by a single separating curve (not homotopic to boundary components), so that $\partial K\cap X'\subset Y$, each puncture of $X'$ is a puncture of $Y$, and
\( \H1(X';\Z) < \H1(Y;\Z). \)\\
Then there is a subsurface $X\subset K$ bounded by a single curve such that $X\subset Y$, $\partial K\cap X=\partial K\cap X'$ and \( \H1(X;\Z) = \H1(X';\Z). \)
\end{lemma}

\begin{proof}
To simplify the notation, replace all punctures by boundary components.

Let  \( \gamma_1, \ldots, \gamma_r \) be the boundary components of \( K \) contained in \( X' \).
Let \( g \) denote the genus of \( X' \).
For \( i \leq g \), choose \( a_i, b_i \in \H1(X'; \Z) \) so that \( \ai(a_i,b_j) = \delta_{ij} \) and \( \ai(a_i,a_j) = \ai(b_i,b_j) = 0 \) for all \( i,j \in \{1, \ldots, g\} \).
Observe that
\[
\H1(X'; \Z) = \Span\{ a_i, b_i, \gamma_j : i \leq g, j \leq r\}.
\]

Let \( \alpha_i \) and \( \beta_i \) be simple closed curves in \( Y \) homologous to \( a_i \) and \( b_i \), respectively, and whose geometric intersection is the same as the algebraic intersection of the corresponding classes.
Fix a surface \( Z \subset Y \) with a single boundary component containing the \( \alpha_i \) and \( \beta_i \) and such that 
\[
\H1(Z; \Z) = \Span\{a_i, b_i : i \leq g\}.
\]
Choose pairwise disjoint simple arcs \( \delta_1, \ldots, \delta_r \) contained in \( Y \smallsetminus Z \) such that \( \delta_i \) connects \( \partial Z \) and \( \gamma_i \).
Define \( X \) to be a regular neighborhood of \( Z \cup \bigcup_{i=1}^r (\gamma_i \cup \delta_i) \). 
By construction, $X$ satisfies the requirements.
\end{proof}

\begin{prop}\label{prop:realize-nested}
  If $X',Y$ are a flare surfaces such that
  \( \H1(X';\Z) < \H1(Y;\Z), \)
  then there is a flare surface $X$ with $\H1(X;\Z) = \H1(X';\Z)$ and $X\subset Y$.
\end{prop}

\begin{proof}
Let $K \subset S$ be a star surface which contains $\partial X' \cup \partial Y$.
Denote by $U_1, \ldots, U_k$ the complementary components of \( K \).
Observe that each $U_i$ is either contained in, or disjoint from, $X'$ as they are disjoint from $\partial X$ (and analogously for $Y$).

Observe that if \( U_i \subset X' \), then \( \H1(U_i; \Z) \leq \H1(X';\Z) \leq \H1(Y;\Z ) \); hence, by Lemma \ref{lem:endset}, \( \lends(U_i) \subset \lends(Y) \).
It follows by the choice of \( K \) that \( U_i \subset Y \).

So, up to reordering, we have:
\[ X' = U_1 \cup \cdots U_r\cup K_{X'} \]
\[ Y = U_1 \cup \dots U_{r+s}\cup K_{Y} \]
where \( K_{X'} = K \cap X' \) and \(  K_{Y} = K \cap Y \). 
Note also that all punctures of $K_{X'}$ are punctures of $K_Y$ as well since $\lends(\partial X')\subset\lends(\partial Y)$.

We now want to show that $ \H1(K_{X'};\Z) < \H1(K_{Y};\Z).$

Since we have seen that $\partial U_j\subset K_Y$ for all $j\leq r$ and that all punctures of $K_{X'}$ are also punctures of $K_Y$, we know that every isotropic vector in $\H1(K_{X'};\Z)$ is also in $\H1(K_{Y};\Z)$.
Look now at any non-isotropic vector $v\in\H1(K_{X'},\Z)<\H1(K;\Z)$. Choose a standard basis for homology of $K$ such that the non-separating curves are either completely contained in $K_Y$ or in $K\ssm K_Y$ and all boundaries of $K_Y$ are part of the basis. If we decompose $v$ with respect to this basis, we get
 $$v=x+y$$
 where $x\in\H1(K_Y;\Z)$ and $y$ is a linear combination of classes of curves in $K\ssm K_Y$. If $y$ were not isotropic, it would have non-zero intersection with some curve in $K\ssm K_Y\subset Y$ and hence so would $v$, a contradiction since $v\in\H1(Y;\Z)$. So
 $$y=\sum_{i=r+s+1}^k c_i [\partial U_i]+\sum_{i=k+1}^{k+p}c_i[\gamma_i],$$
where $p$ is the number of punctures in $K\ssm K_Y$ and each $\gamma_i$ is a curve surrounding one puncture of $K\ssm K_Y$ (and leaving it to the right).
 
If all  $c_i$ are the the same, then $y$ is a multiple of $\partial Y$ and hence belongs to $\H1(K_Y;\Z)$ and so does $v$. Otherwise there is an arc $\alpha\in K\ssm K_Y$ that intersects $y$ non-trivially and hence it intersects $v$ non-trivially, a contradiction.

So also all non-isotropic vectors of $ \H1(K_{X'};\Z)$ belong to $\H1(K_{Y};\Z)$, which shows that $ \H1(K_{X'};\Z) < \H1(K_{Y};\Z)$.

Note that all boundary components of $K$ that are in $K_{X'}$ are in $K_{Y}$ as well. This implies that we can apply Lemma~\ref{lem:nested-finite-type} to find a subsurface $K_X \subset K$ cut off by a single curve, contained in $K_{Y}$, with $\partial K\cap K_X=\partial U_1\cup\dots\cup\partial U_r$ and
\[ \H1(K_X;\Z) = \H1(K_{X'};\Z). \]
Hence
\[ X = U_1 \cup \cdots U_r \cup K_X \]
is the desired subsurface.
\end{proof}

\subsection{The homology end filtration and its ultrafilters}\label{sec:ultrafilters}

The following is the central object of this section.
\begin{defin}\label{defin:filt}
We define
$$\filt=\{V< \H1(S;\Z) \st V=\H1(X,\Z) \mbox{ for some } X\in\hs\}$$
and for every $e\in\Ends(S)$ we define $\filt_e\subset\filt$ to be 
$$\filt_e=\{V< \H1(S;\Z) \st V=\H1(X,\Z) \mbox{ for some } X\in\hs \mbox{  with } e\in\lends(\partial X)\}.$$
We call $\filt$ the \emph{homology end filtration} and we say that an automorphism of $\H1(S;\Z)$ \emph{preserves $\filt$} if it induces a permutation of $\filt$.
\end{defin}
We emphasize that $\filt$ contains only the homology group, without the data
of which flare surface yielded the group.
Note that $\filt$ is endowed with a natural partial order given by inclusion and if $\phi$ is an automorphism of \( \H1(S;\Z) \) preserving the homology end filtration, then it induces an automorphism of $\filt$ as a poset.

We first want to show that if an automorphism of \( \H1(S;\Z) \) preserves the homology end filtration, then it induces a permutation of the set $\{\filt_e\st  e\in\Ends(S)\}$.
This will allow us to define an associated map of the space of ends.

To get the result, we will show how these subsets  of $\filt$ correspond to ultrafilters in $\filt$.
Recall that, if $(P,\leq)$ is a poset, a \emph{filter} is a non-empty subset $F$ of $P$ such that:
\begin{enumerate}
\item for all $ x,y\in F$ there exists $z\in F$ with $z\leq x, z\leq y$;
\item if $x\in F$ and $x\leq y$, then $y\in F$.
\end{enumerate}
A filter $U$ is called an \emph{ultrafilter} if it is a maximal proper filter of $P$, that is, \( U \neq P \) and if  $F$ is a proper filter such that $U\subseteq F$, then \( F = U \).

First, we discuss the homology end filtration and its ultrafilters.

\begin{lemma}\label{lem:ultrafilters}
Given a surface \(S\) satisfying \( (\star) \), $U$ is an ultrafilter if and only if $U=\filt_e$  for some $e\in\Ends(S)$.
\end{lemma}
\begin{proof}
We show first that for every $e$, $\filt_e$ is an ultrafilter.

Let $V,W\in\filt_e$ and let \( X \) and \( Y \) be flare surfaces such that $V=\H1(X;\Z)$ and $W=\H1(Y;\Z)$.
The intersection $X\cap Y$ contains a flare surface -- say $T$ -- with $e$ as an end. 
Then $\H1(T;\Z)\in\filt_e$ and $\H1(T;\Z)\leq V$, $\H1(T;\Z)\leq W$.
So property (1) of a filter holds.

Property~(2) follows from Lemma \ref{lem:endset}.

Finally, suppose there exists a proper filter $U$ containing $\filt_e$.
Let $V \in U \smallsetminus \filt_e$ and choose a flare-surface $X$ so that $\H1(X;\Z)=V$ and hence $e\notin\lends(\partial {X})$. By the assumption on the topology of \( S \), we can find a flare surface $Y$ containing $e$ and disjoint from $X$.
Property~(1) of a filter guarantees then the existence of a flare surface $Z$ such that
\[ \H1(Z;\Z) \subset \H1(X;\Z)\cap \H1(Y;\Z) \]
contradicting Lemma \ref{lem:endset}.

Conversely, let $U$ be an ultrafilter and consider \( \lends_U = \{\lends(\partial X) : \H1(X;\Z) \in U \} \).
Property (1) of filters together with Lemma \ref{lem:endset} implies that the intersection of any finite collection of sets in \( \lends_U \) is non-empty (i.e.~ \( \lends_U \) has the finite intersection property). 
Hence, as each element of \( \lends_U \) is closed and \( \Ends(S) \) is compact, the intersection 
\(
\bigcap_{C \in \lends_U} C
\)
is non-empty.
If \( e \) is an element in the intersection, then \( U \subset \filt_e \); hence, by maximality, \( U = \filt_e \).
\end{proof}

\subsection{A homeomorphisms of the space of ends}

Let \( \mathcal{U}(\filt) \) be the set of ultrafilters of \( \filt \) and, for each \( V \in \filt \), let \[ N_V = \{ U \in \mathcal{U}(\filt) \st V \in U \}. \] 
We define a topology on \( \mathcal{U}(\filt) \) by declaring the sets of the form \( N_V \) to be a basis.
By Lemma \ref{lem:ultrafilters} and since different ends define different ultrafilters, we have a bijective map \( \theta \co \Ends(S) \to \mathcal{U}(\filt) \) defined by \( \theta(e) = \filt_e \). 

\begin{lemma}
For a surface $S$ satisfying \( (\star) \), the map
\( \theta \) is a homeomorphism. 
\end{lemma}

\begin{proof}
Fix \( V \in \filt \) and let \( N = N_V \).
We can then choose \( X \in \hs \) such that \( V = \H1(X; \Z) \). 
Tracing definitions, we have that 
\begin{align*}
\theta^{-1}(N) & = \{ e \in \Ends(S) \st \filt_e \in N \}\\
& = \{ e \in \Ends(S) \st V \in \filt_e\} \\
& = \{ e \in \Ends(S) \st e \in \lends(X)\} \\
& = \lends(X),
\end{align*}
where the third equality is a consequence of Lemma \ref{lem:endset}.
Therefore, \( \theta \) is a continuous bijective map; moreover, the above chain of equalities (in reverse) shows that \( \theta \) is an open map and hence a homeomorphism.
\end{proof}

Ultrafilters of \( \filt \) are preserved under poset automorphisms of \( \filt \) and, as a consequence, any such automorphism of \( \filt \) will induce a homeomorphism of \( \mathcal{U}(\filt) \).
Consequently, given an automorphism \( \phi \) of \( \H1(S; \Z) \) preserving the homology end filtration \( \filt \), we can define the homeomorphism \( f_\phi \co \Ends(S) \to \Ends(S) \) by \( f_\phi(e) = \theta \circ \hat\phi \circ \theta^{-1} \), where \( \hat \phi \co \mathcal{U}(\filt) \to \mathcal{U}(\filt) \) is the homeomorphism of \( \mathcal{U}(\filt) \) defined by \( \hat \phi(N_V) = N_{\phi(V)} \).
We record this in the following proposition:

\begin{prop}
\label{prop:end-homeo}
Let \( S \) be a surface satisfying \( (\star) \).
An automorphism \( \phi \) of \( \H1(S;\Z) \) preserving the homology end filtration induces a homeomorphism \( f_\phi \) of \( \Ends(S) \) defined by the property \( \filt_{f_\phi(e)} = \hat{\phi}(\filt_e)  \). \qed
\end{prop}

\subsection{The homology end filtration and simple isotropics}\label{sec:simpleis}
As we saw in the last section, an automorphism of the homology end filtration induces a homeomorphism on the space of ends.
Given the correspondence between simple isotropics and clopen subsets of the end space, we expect that any automorphism of the homology end filtration must preserve the set of simple isotropics; indeed:

\begin{prop}
\label{prop:simple-preserving}
Let \( S \) be a surface satisfying \( (\star) \).
If \( \phi \) is an automorphism of \( \H1(S; \Z) \) preserving the homology end filtration, then \( \phi \) preserves the set of simple isotropic elements of \( \H1(S; \Z) \). 
\end{prop}

To prove this, we need the following lemma:

\begin{lemma}\label{lem:filt_isolated}
Let \( S \) be a surface satisfying \( (\star) \) and let $e\in \Ends(S)$.  If $e$ is isolated, then
$$\bigcap_{V\in \filt_e}V=\Span(c)$$
where $c$ is a simple isotropic with $\lends(c)=\{e\}$.
Moreover, \( e \) is isolated if and only if $$\bigcap_{V\in \filt_e}V\neq \{0\}.$$
\end{lemma}

\begin{proof}
Suppose first that $e$ is isolated. Let $V\in\filt_e$ and let $X$ be a flare surface such that $V=\H1(X;\Z)$ and $e\in\lends(\partial X)$. Then there is a separating simple closed curve $\alpha\subset X$ such that $\lends(\alpha)=\{e\}$. Then by Lemma \ref{lem:twosep}, $c=[\alpha]$ and thus $\Span(c)\subset V$.
To show that the intersection is not bigger than the span of $c$, one can easily construct explicit flare surfaces whose homologies intersect in the span.

Suppose now that $e$ is not isolated and let $x\in\bigcap\{V \st V\in\filt_e\}$. Then $x$ can be realized by some loop $\alpha$ in a compact subsurface $K$ with separating boundary curves. Let $X$ be a flare surface containing $e$ with $|\Ends(S)\ssm\lends(\partial X)|\geq 2$ and disjoint from $K$. Furthermore, let $Y\subset X$ be another flare surface with $\lends(\partial X)\ssm\lends(\partial Y)\neq \emptyset$ and \( e \in \lends(\partial Y) \). Then $x\in \H1(S\ssm X;\Z)$, because we can realize it in $K$, and $x\in\H1(Y;\Z)$ since $x$ is in the homology of all flare surfaces containing $e$. By Lemma \ref{lem:flaredisjoint}, $\H1(S\ssm X;\Z)\cap\H1(Y;\Z)=\{0\}$; hence,  $x=0$.
\end{proof}

\begin{proof}[Proof of Proposition \ref{prop:simple-preserving}]
Let \( a \in \H1(S; \Z) \) be a simple isotropic. 
First suppose that \( S \) is Jacob's ladder.
In this case, \( \pm a \) are the unique primitive homology classes contained in the homology of every flare surface; hence, \( \phi(a) = \pm a \) and \( \phi \) preserves the unique simple isotropic element.

We can now assume that \( S \) is not Jacob's ladder.
Now suppose first that neither $\lends(a)$ nor $\rends(a)$ is a single isolated puncture. Then there is a flare surface \( X \) with \( [\partial X] = a \).
Let \( Y \) be the closure of the complement of \( X \) in \( S \), so that \( Y \) is a flare surface satisfying \( [\partial Y] = -a \). 
The intersection \( \H1(X; \Z) \cap \H1(Y; \Z) \) is generated by \( a \).
Note that for every end \( e \) of \( S \), either \( \H1(X; \Z) \in \filt_e \) and \( \H1(Y;\Z) \notin \filt_e \) or vice versa. 

As \( \phi \) induces a homeomorphism on the space of ends (Proposition \ref{prop:end-homeo}), it follows that for every end \( e \) of \( S \), either \( \phi(\H1(X; \Z)) \in \filt_e \) and \( \phi(\H1(Y;\Z)) \notin \filt_e \) or vice versa. 
Let \( X' \) and \( Y' \) be such that \( \H1(X';\Z) = \phi(\H1(X; \Z)) \) and \( \H1(Y'; \Z) = \phi(\H1(Y;\Z)) \). 
We know that \( \H1(X'; \Z) \cap \H1(Y'; \Z) \) is cyclic and generated by \( \phi(a) \).
Further, \( \Ends(S) = \lends(\partial X') \sqcup \lends(\partial Y') \) implying that \( [\partial Y'] = -[\partial X'] \).
It follows from Lemma \ref{lem:outside-curves} that in fact \( \H1(X'; \Z) \cap \H1(Y';\Z) \) is generated by \( [\partial X'] \).
Therefore, \( \phi(a) = \pm [\partial X'] \) and hence is a simple isotropic.

If $\lends(a)$ or $\rends(a)$ is a single isolated puncture, say $p$, then
$$\Span(a)=\bigcap_{V\in\filt_p}V,$$
hence
$$\Span(\phi(a))=\bigcap_{V\in\filt_{f_\phi(p)}}V.$$
By Lemma \ref{lem:filt_isolated}, $\phi(a)$ is a simple isotropic.
\end{proof}

The proof of Proposition \ref{prop:simple-preserving} yields two additional lemmas that we record.

\begin{lemma}\label{lem:boundariestoboundariesv1}
Let \( S \) be a surface satisfying \( (\star) \) and 
let \( \phi \) be an automorphism of \( \H1(S; \Z) \) preserving the homology end filtration.
If $X$ and $Y$ are flare surfaces satisfying $$\phi(\H1(X;\Z))=\H1(Y;\Z),$$ then $\phi([\partial X])=\pm[\partial Y]$. \qed
\end{lemma}

\begin{lemma}
\label{lem:single-coherence}
Let \( S \) be a surface satisfying \( (\star) \).
If \( \phi \) is an automorphism of \( \H1(S; \Z) \) preserving the homology end filtration and \( a \) is a simple isotropic, then either \[ \lends(\phi(a)) = f_\phi(\lends(a)) \] or \[ \rends(\phi(a)) = f_\phi(\rends(a)). \] \qed
\end{lemma}

In the next lemma, we see that the homeomorphism of the space of ends induced by an automorphism of the homology end filtration either preserves the notion of ``to the left" or reverses it coherently across all simple isotropics. 

\begin{lemma}
\label{lem:preserving-is}
Let \( S \) be a surface satisfying \( (\star) \) and let  \( \phi \) be an automorphism of \( \H1(S; \Z) \) preserving the homology end filtration.
If there exists a simple isotropic \( a \) such that \( \lends(\phi(a)) = f_\phi(\lends(a)) \), then \( \lends(\phi(b)) = f_\phi(\lends(b)) \) for every simple isotropic \( b \) in \( \H1(S;\Z) \). 
\end{lemma}

\begin{proof}
We proceed by contradiction: suppose there exists \( b \) such that \( f_\phi(\lends(b)) = \rends(\phi(b)) \).
We have two cases: either \( \lends(a) \cap \lends(b) = \emptyset \) or \( \lends(a) \cap \lends(b) \neq \emptyset \).
In the first case, \( a+b \) is a simple istropic.
It follows that in this case \( \lends(\phi(a)) \cap \lends(\phi(b)) \neq \emptyset \), but then \( \phi(a+b) = \phi(a)+\phi(b) \) is not a simple isotropic, which contradicts Proposition \ref{prop:simple-preserving}.

In the second case, we choose simple isotropics \( a' \) and \( b' \) such that \( \lends(a') = \lends(a) \ssm \lends(b) \) and \( \lends(b') = \lends(b) \ssm \lends(a) \).
Observe that \( a-a' \) is a simple isotropic.
If \( f_\phi(\lends(a')) = \rends(\phi(a')) \), then \[ \rends(\phi(a')) = f_\phi(\lends(a')) \subset f_\phi(\lends(a)) = \lends(\phi(a)) ;\] hence, \( \phi(a-a') = \phi(a)-\phi(a') \) is not a simple isotropic, again contradicting Proposition \ref{prop:simple-preserving}.
Therefore, \( f_\phi(\lends(a')) = \lends(\phi(a')) \) and, by a similar argument, \( f_\phi(\lends(b')) = \lends(\phi(b')) \).
This puts us back in the first case and arriving at another contradiction.
\end{proof}

Together, Lemma \ref{lem:boundariestoboundariesv1} and Lemma \ref{lem:preserving-is}, yield the following:

\begin{lemma}\label{lem:boundariestoboundaries}
Let \( S \) be a surface satisfying \( (\star) \) and  let  \( \phi \) be an automorphism of \( \H1(S;\Z) \) preserving the homology end filtration.
If $X$ and $Y$ are flare surfaces satisfying $$\phi(\H1(X;\Z))=\H1(Y;\Z),$$ then $\phi([\partial X])=[\partial Y]$. \qed
\end{lemma}

We now strengthen Proposition \ref{prop:end-homeo} by detecting the topological types of ends.

\begin{prop}
\label{prop:end-homeo-planar}
Let \( S \) be a surface satisfying \( (\star) \).
If \( \phi \) is an automorphism of \( \H1(S;\Z) \) preserving the homology end filtration, then the homeomorphism \( f_\phi \) of \( \Ends(S) \) induced by \( \phi \) preserves the set of planar ends.
\end{prop}

\begin{proof}
Observe that \( e \) is planar if and only if there exists \( H \in \filt_e \) such that \( H \) is an isotropic subspace of \( \H1(S; \Z) \).
By Proposition \ref{prop:simple-preserving},  we have that \( \phi(H) \) is isotropic if and only if \( H \) is isotropic; hence if \( e \) is planar, then so is \( f_\phi(e) \). 
\end{proof}

\subsection{Characterizing the image of $\rho_S$ using simple isotropic elements}\label{sec:charwithis}
As we have seen that if a map preserves the filtration, then it preserves the set of simple isotropic vectors, it is natural to ask whether the converse holds.
More generally, we can ask if the homology end filtration requirement in Theorem \ref{mainthm} can be replaced by a condition on the action of the automorphism on the simple isotropic classes.

By considering a finite-type surface with genus $g$ and $n$ punctures (for $n\geq3$), it is easy to see that preserving the algebraic intersection form and the set of simple isotropic classes does not suffice to guarantee realizability: a mapping class cannot send a curve surrounding a single puncture to a curve surrounding two.

Nevertheless, in certain cases one can use  the set of simple isotropic classes to describe the image of the homology representation.
Namely, we can define a partial order on this set, where we say that $[\gamma]\leq [\delta]$ if $\lends(\gamma)\subset \lends(\delta)$. Lemma \ref{lem:twosep} shows that this order is well defined.
If we ask that an automorphism preserves $\ai$ and the poset of simple isotropic classes, the example mentioned above is ruled out.
In fact, one could show -- with techniques similar to those we use -- that this is enough to characterize automorphisms induced by mapping classes in the case of surfaces with at least two ends and at most one non-planar end. On the other hand, as soon as there is more than one non-planar end, this characterization does not hold: for instance, consider Jacob's ladder with the basis depicted in Figure \ref{fig:basisJacob}. The automorphism fixing all non-isotropic basis vectors and sending $[\gamma]$ to $-[\gamma]$ preserves $\ai$ and the poset of simple isotropic classes, but it cannot be given by a mapping class.

\section{Proof of the main theorem}\label{sec:proof}
The goal of this section is to prove Theorem \ref{mainthm}, characterizing the image of the mapping class group in the group of automorphisms of $\H1(S;\Z)$. We will show:
\begin{theor}\label{thm:iff}
Let \( S \) be either a planar surface with at least four ends, of finite positive genus with at least three ends,  or an infinite-genus surface different from the Loch Ness monster and the once-punctured Loch Ness monster.
The image of $\rho_S$ is the group of automorphisms of $\H1(S;\Z)$ preserving both $\ai$ and $\filt$ and such that there is a simple isotropic class $c$ for which
$$f_\phi(\lends(c))=\lends(\phi(c)).$$
\end{theor}

Using this, the main theorem (which we now recall) easily follows.\\

\begin{duplicate}[\ref{mainthm}]
  Let $S$ be either a finite-type surface with at least four punctures or an infinite-type surface different from the Loch Ness monster and the once-punctured Loch Ness monster. If
  $\phi$ is an automorphism of $\H1(S;\Z)$ preserving both $\ai$ and $\filt$,
  then the following hold:
  \begin{enumerate}[i)]
  \item Exactly one of $\phi$ and $-\phi$ lies in the image of $\rho_S$.
  \item $\phi$ preserves homology classes defined by separating simple
    closed curves.
  \item $\phi$ determines a homeomorphism $f_\phi$ of the space of ends of $S$, 
    and $\phi$ lies in the image of $\rho_S$ exactly if
    \[ f_\phi(\lends([\delta])) = \lends(\phi([\delta])) \]
    for some (hence any) simple separating closed curve $\delta$ which is non-trivial in 
    $H_1(S;\Z)$.
  \end{enumerate}
\end{duplicate}

\begin{proof}
Part ii) follows from Proposition \ref{prop:simple-preserving} and Theorem \ref{thm:iff}.

The fact that $\phi$ induces a homeomorphism of the space of ends is given by Proposition \ref{prop:end-homeo}. This together with Theorem \ref{thm:iff} yields part iii).

Part i) follows from iii) using Lemma \ref{lem:preserving-is}.
\end{proof}

Let us then prove Theorem \ref{thm:iff}.

\begin{proof}[Proof of Theorem \ref{thm:iff}]
  We prove the theorem for infinite-type surfaces. In the case of a finite-type surface, we just need to adapt the base case below.

  It is clear that any mapping class induces an automorphism
  satisfying the conditions in the statement, so we want to show that
  an automorphims with these properties is induced by a mapping class.
 Let $\phi$ be such an automorphism.

  Fix an exhaustion $\Sigma_k$ of $S$ by star surfaces such that no
    component of \( \partial \Sigma_{k+1} \) is homotopic to any
    component of \( \partial \Sigma_k \).

  The goal is to construct two nested sequences of star surfaces
  $\{A_k\}$ and $\{B_k\}$, together with homeomorphisms $f_k\co A_k\to B_k$ such
  that:
  \begin{enumerate}
  \item $\Sigma_k\subset A_k$ for every odd $k$ and $\Sigma_k\subset
    B_k$ for every even $k$;
  \item \( f_{k}|_{A_{k-1}} = f_{k-1} \);
  \item $f_k$ induces $\phi\big|_{\H1(A_k;\Z)}$.
  \end{enumerate}
  As in the proof of Theorem \ref{monster}, condition (2)
  implies the direct limit of the \( f _n \) exists and condition (1)
  implies that both sequences are exhaustions and hence
  $\displaystyle{f=\varinjlim f_n}$ is a homeomorphism of $S$.
  Condition (3) then guarantees that \( f \) acts on homology as
  $\phi$.

  We will construct surfaces and maps satisfying the additional
  condition:
  \begin{enumerate}[(4)]
  \item For every component $X$ of $S\ssm A_k$ the following
    holds: if $Y$ is the component of $S\ssm B_k$ bounded by
    $f_k(\partial X)$, then $\phi(\H1(X;\Z))=
    \phi(\H1(Y;\Z))$.
  \end{enumerate}

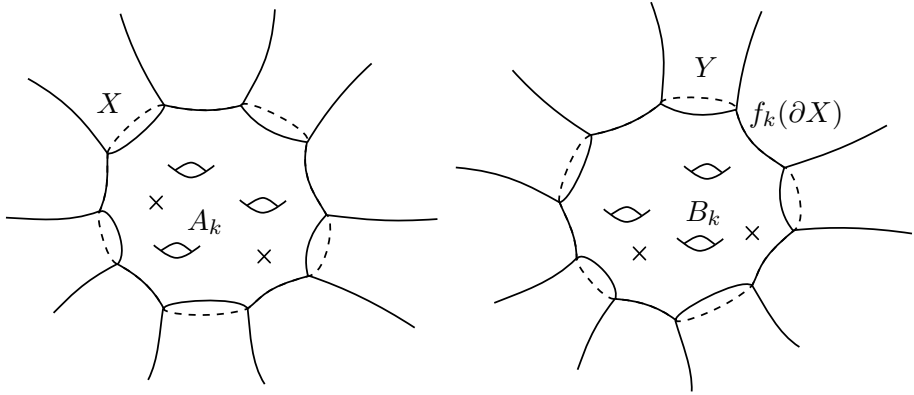
\begin{figure}[h]
\begin{overpic}[width=.8\textwidth]{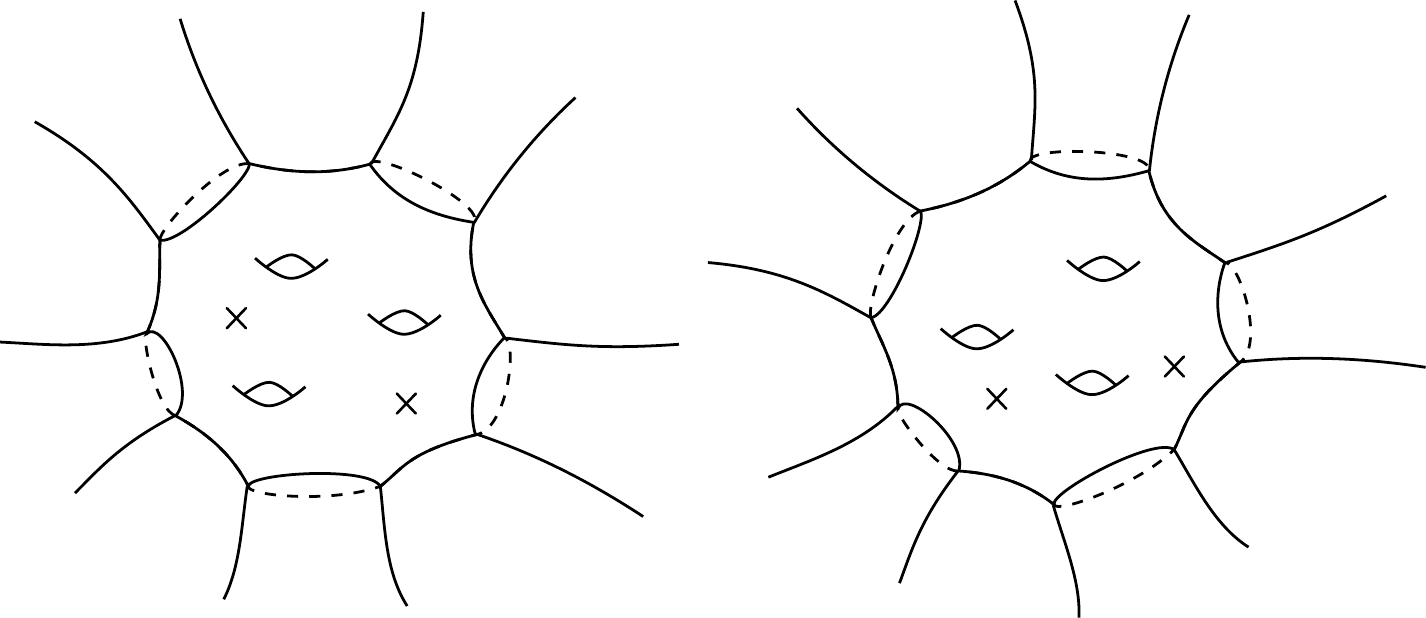}
\put(20,18){$A_k$}
\put(10,31){$X$}
\put(75,19){$B_k$}
\put(82,30){$f_k(\partial X)$}
\put(76,35){$Y$}
\end{overpic}
\caption{The subsurfaces in condition (4)}
\end{figure}

\textbf{Base case:}
Let $A_1 = \Sigma_l$ where $i_1$ is the first index so that $\Sigma_i$ has either more than one boundary component,
or contains more than one puncture. That such an index exists follows from the fact that $S$ is neither the Loch Ness
nor the once-punctured Loch Ness surface. Let $g_1$ be the genus of $A_1$.

Choose $2g_1$ non-separating curves $\alpha_1, \beta_1,\dots,\alpha_{g_1}, \beta_{g_1}$ with the standard symplectic intersection pattern.
Realize the classes $\phi([\alpha_i]),\phi([\beta_i])$ by non-separating curves $\alpha_i',\beta_i'$ with the standard symplectic intersection pattern (if $g_1=0$, we do not do anything in this step).

Choose a subsurface $F_1$ of genus $g_1$ with one boundary component and containing the $\alpha_i',\beta_i'$ as well as the images via $f_\phi$ of the punctures of $A_1$ (if $g_1=0$ and $A_1$ has no punctures, just set $F_1=\emptyset$).

Denote by $X_1, \ldots, X_b$ the complementary components of $A_1=\Sigma_1$.
By construction, for every \( j \in \{1, \ldots, b \} \), we have \( \phi(\H1(X_j; \Z)) \subset \H1( S \ssm F_1 ; \Z ) \).
Moreover, if \( j \neq j' \), then \( \phi( \H1(X_j ; \Z)) \cap \phi ( X_{j'} ; \Z ) \) is trivial, unless \( b =2  \), in which case the intersection is generated by \( \phi([\partial X_1]) \).
The goal is to realize these homology groups by disjoint flare surfaces:
By Proposition \ref{prop:realize-nested} there exists a flare surface \( Y_1 \) contained in \( S \ssm F_1 \) such that $\H1(Y_1 ; \Z ) = \phi(\H1(X_1;\Z))$.
Choose a simple arc \( \eta \) in the closure of \( S \ssm F_1 \) connecting \( \partial F_1 \) and \( \partial Y_1 \) and define \( F_2 \) to be a regular neighborhood of \( F_1 \cup \eta \cup Y_1 \)  (if $g_1=0$ and $A_1$ has no punctures, just set $F_2=Y_1$).
As $\phi(\H1(X_2,\Z))\subset \H1(S\ssm F_2;\Z)$, we can find a second flare surface $Y_2\subset S\ssm F_2$ with $\phi(\H1(X_2;\Z))=\H1(Y_2;\Z)$.
Repeating this process, we obtain \( Y_1, \ldots, Y_b \) such that \( Y_j \cap Y_{j'} = \emptyset \) whenever \( j \neq j' \) and such that \( \H1( Y_j; \Z ) = \phi( \H1( X_j ; \Z ) ) \) for every \( j \in \{1, \ldots, b\} \).

\begin{figure}[t]
\begin{overpic}{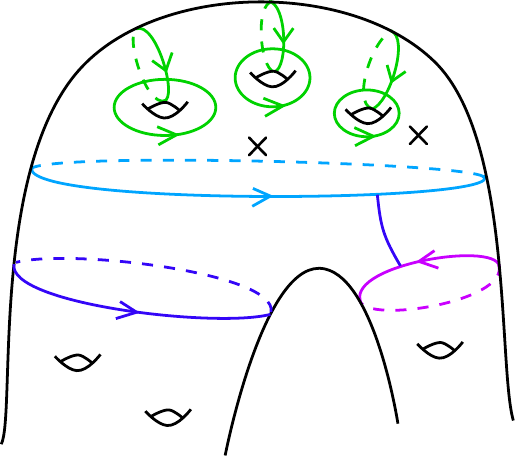}
\put(45,42){$\partial F_1$}
\put(65,40){$\eta$}
\put(100,34){$\partial Y_1$}
\put(30,18){$\partial F_2$}
\end{overpic}
\caption{Constructing the $Y_j$}
\end{figure}
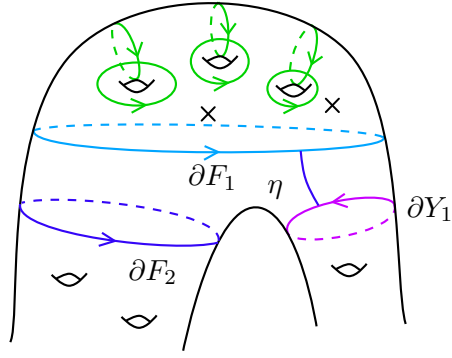

Define \( B_1 = S \ssm \left( \bigcup_{j=1}^b Y_j \right) \).
Then, as \( \lends( \partial Y_j ) = f_\phi(\lends( \partial X_j ) ) \) by Lemma \ref{lem:boundariestoboundaries}, we have 
\[
\bigcup \lends(\partial Y_i )\cup f_{\phi}(\{\mbox{punctures of }A_1\})=\Ends(S)
\]
and hence \( B_1 \)  is a star surface with \( b \) boundary components and as many punctures as $A_1$.
Now, since \( B_1 \) and \( A_1 \) have the same number of boundary components, the same number of punctures, and isomorphic homology, we can conclude that \( B_1 \) is homeomorphic to \( A_1 \).
Now choose a homeomorphism $f_1\co A_1\to B_1$ mapping $\alpha_i,\beta_i,\partial X_j$ to $\alpha_i',\beta_i',\partial Y_j$, respectively, and agreeing with $f_\phi$ on the punctures of $A_1$. 
By construction, the triple \( (A_1, B_1, f_1) \) satisfies conditions (1)-(4). 

\textbf{Induction step:} Suppose we have $A_k,B_k$ and $f_k$ satisfying conditions (1)-(4).

If $k$ is even, let $K\geq k+1$ be such that $A_k\subsetneq \Sigma_K$ and set $A_{k+1}:=\Sigma_K$.

Let $X$ be a complementary component of $A_k$, and let $Y$ be the complementary component of $B_k$ bounded by $f_k(\partial X)$. Let $q$ be the genus of $X \cap A_{k+1}$. Choose curves $\alpha_1, \beta_1, \dots,\alpha_q, \beta_q$ in $X\cap A_{k+1}$ with the standard symplectic intersection pattern.

By condition (4), we can realize the classes $\phi([\alpha_1]),\dots,\phi([\beta_q])$ by curves  $\alpha_1',\dots,\beta_q'$ in $Y$ with the standard intersection pattern.
Choose a separating curve in \( Y \) bounding a surface \( F \) of genus \( q \) containing \( \alpha_1',\beta_1', \dots,\alpha_q', \beta_q'\) and the images of the punctures of \( X \cap A_{k+1} \) under \( f_\phi \).

Let $X_1, \ldots, X_r$ be the flare surfaces $X_i$ which are the
components of $X\ssm A_{k+1}$.

Arguing as in the base case, we can find disjoint flare surfaces $Y_i$
in $Y$ so that $\phi(\H1(X_i;\Z))=\H1(Y_i;\Z)$ for all $i$.  The
boundaries $\partial Y_i$, together with $\partial Y$, cut off a
compact subsurface $K_X\subset Y$ homeomorphic to $X\cap A_{k+1}$.

We can therefore choose a homeomorphisms $f_X^{k+1}\co X\cap
A_{k+1}\to K_X$ sending $\alpha_1,\beta_1, \dots,\alpha_q, \beta_q$ to
$\alpha_1',\beta_1',\dots,\alpha_q', \beta_q'$, agreeing with $f_k$ on $\partial X$
and with $f_{\phi}$ on the set of punctures of $X\cap A_{k+1}$, and
sending $\partial X_i$ to $\partial Y_i$ for all $i$.

Since all complementary flare surfaces $X$ of $A_k$ are disjoint, we can repeat this process
independently on all of them, obtaining sets $K_X$ and maps $f^{k+1}_X$.
Now let $B_{k+1}$ be the union of $B_k$ with all  $K_X$, that is,
\[
B_{k+1} = B_k \cup \left( \bigcup_{X\in \pi_0(S \ssm A_k)} K_X \right).
\] 
We form $f_{k+1}$ by
gluing the maps $f^{k+1}_X$ to $f_k$:
\[ f_{k+1} = f_k \cup \left(\bigcup f_X^{k+1}\right), \]
which is possible since $f_k, f_X^{k+1}$ have pairwise disjoint supports.
From the construction it immediately follows that $f_{k+1}$ is a homeomorphism between $A_{k+1}$ and $B_{k+1}$ which has the desired properties.

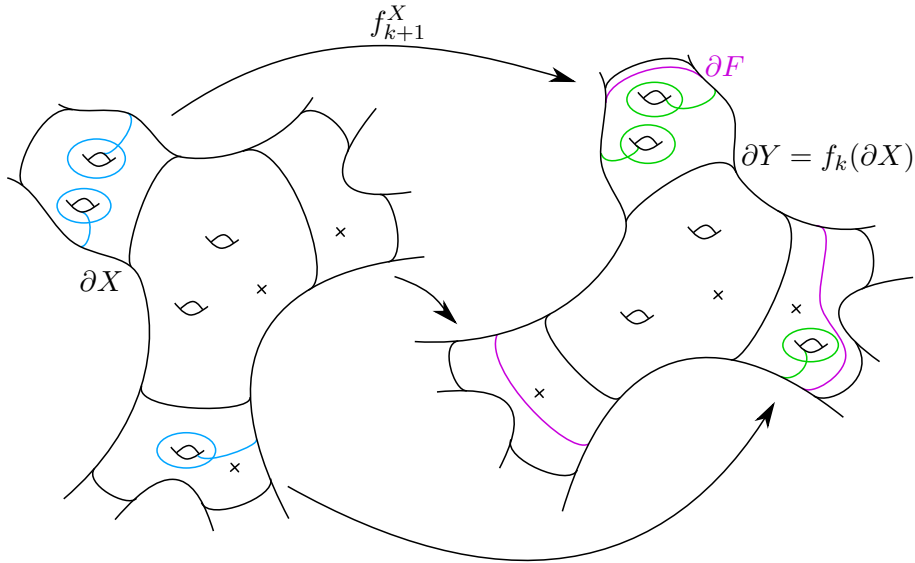
\begin{figure}[t]
\vspace{.5cm}
\begin{overpic}[width=.8\textwidth]{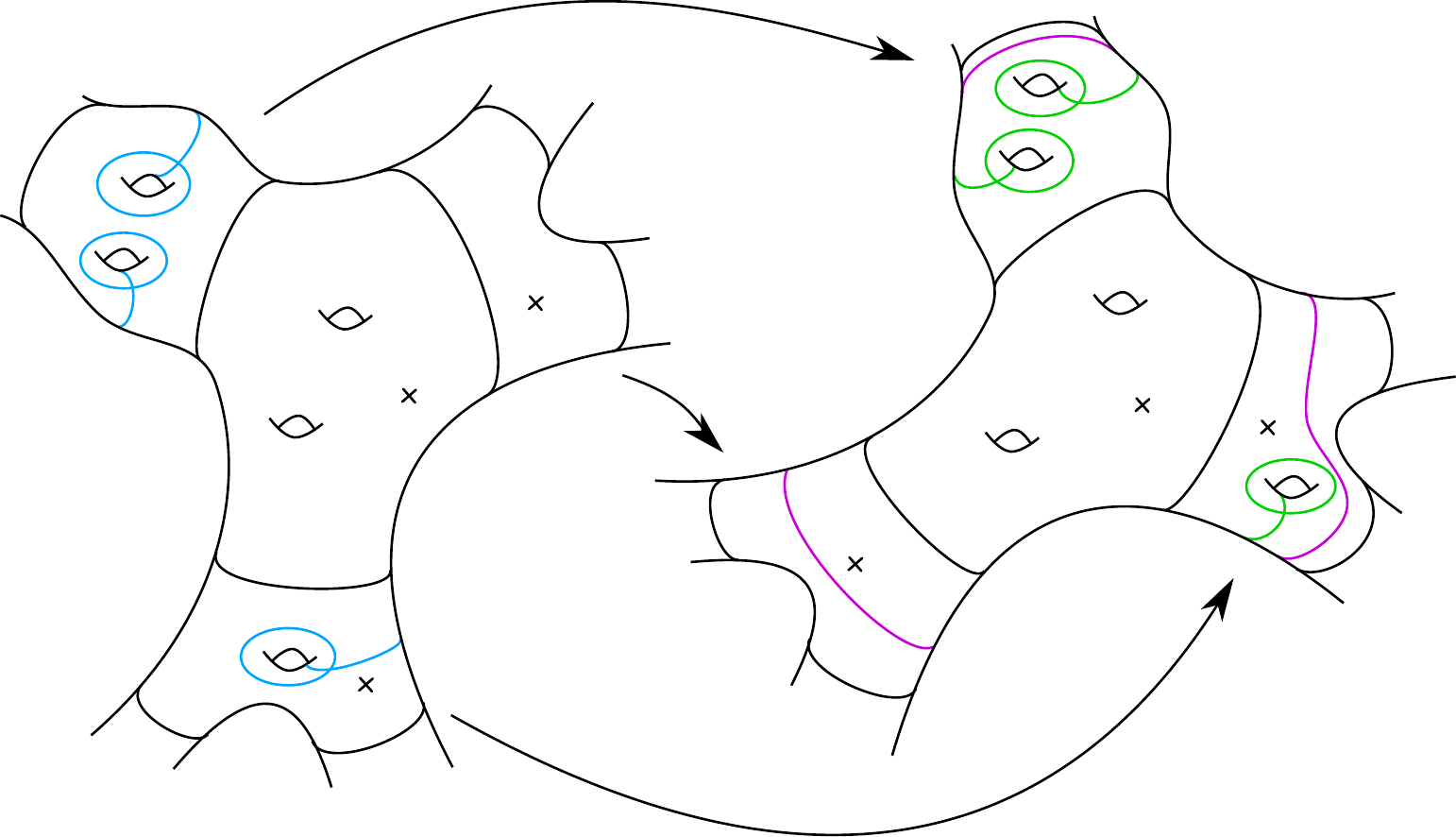}
\put(8,30){$\partial X$}
\put(40,59){$f^X_ {k+1}$}
\put(77,54){\textcolor[RGB]{200,0,220}{$\partial F$}}
\put(81,44){$\partial Y=f_k(\partial X)$}
\end{overpic}
\vspace{.5cm}
\caption{From $k$ to $k+1$ in the proof of Theorem \ref{mainthm}}
\end{figure}

When $k$ is odd we proceed similarly, but we first define $B_{k+1}$ and the curves there and then use $\phi^{-1}$ to get curves outside $A_k$ and hence $A_{k+1}$.
\end{proof}

\appendix
\section{(Co)Homology Classes, Curves and Arcs}\label{sec:curves}
In this appendix we discuss and prove various results describing relations amongst simple curves, simple arcs, and (co)homology classes. 
In the case of finite-type surfaces, most of these results are
well-known;  we collect here extensions to the infinite-type setting, as well as some new 
characterizations.

\smallskip
Let us recall the notation we will use here. We say that a homology class $x\in \H1(S,\Z)$
is \emph{realized by a simple closed curve} if there is a simple closed curve $\gamma$ so that
$[\gamma] = x$. We also say that $x$ is a \emph{simple (non-)isotropic} if $x$ is realized
by a simple closed curve and x is (non-)isotropic.
 
In the case of finite-type surfaces, the characterization of simple (non-)isotropics was done by Meeks and Patrusky \cite{mp_representing}.

Their answer requires the following construction, which is also useful in the
study of infinite-type surfaces.
Given a surface $S$, denote by $\hat{S}$ the surface obtained by filling in the 
planar ends of $S$, and gluing disks to all boundary components. Note that $\hat{S}$ is
compact if $S$ has finite genus.

Let $i\co S\to\hat{S}$ be the natural inclusion and $i_*$ the corresponding map induced 
on first homology. Observe that $\ker(i_*)$ is isotropic with respect to the
algebraic intersection pairing $\ai$ on $\H1(S;\Z)$, and therefore we have
\[ \ai(i_*x, i_*y) = \ai(x,y) \quad\forall x,y\in\H1(S;\Z). \]

In our language, we can now state the characterisation of simple (non-)isotropics for finite-type surfaces as follows.
\begin{theor}[{\cite[Theorem~1]{mp_representing}}]\label{thm:mp}
	Let $S$ be a finite-type surface with $n+1$ punctures and let $\gamma_1,\dots \gamma_{n+1}$ be disjoint curves surrounding the punctures, oriented so that the puncture is to the right.
	 Let 
	$x\in \H1(S;\Z)$ be a non-zero class.  
	\begin{enumerate}[i)]
		\item $x$ is a simple non-isotropic if and only if $i_*x\in \H1(\hat{S};\Z)$ is a non-zero primitive class.
		\item $x$ is a simple isotropic if and only if $x=\pm \sum_{i=1}^{n}\varepsilon_{i}[\gamma_i]$, where $\varepsilon_i\in\{0,1\}$.
	\end{enumerate}
\end{theor}
Observe in particular that the homology class of a separating simple closed curve $\delta$ is completely determined
by the set of punctures lying to the left of $\delta$, which is a special case of Lemma \ref{lem:twosep}.

\smallskip 
In the subsequent parts of this section, we will develop analogous characterizations of simple (non-)isotropics
for infinite-type surfaces. The characterization of simple isotropics involves (algebraic) intersections with
simple arcs joining two ends, and so we also characterize these in the final subsection.

\subsection{Geometric homology bases}\label{sec:geomhombases}
In the case of surfaces of finite type, there are standard bases for homology that are considered; in particular, those given by one curve for all but one puncture and a geometric symplectic basis for the compactified surface. We want to describe standard bases for infinite-type surfaces as well.

A \emph{geometric homology basis} for a surface $S$ is a basis of homology $\{[\alpha_i],[\beta_i]\}_{ i\in I}\cup\{[\gamma_j]\}_{j\in J}$, where the $\alpha_i,\beta_i,\gamma_j$ are all simple closed curves and such that
\begin{itemize}
	\item $\{[\gamma_j]\}_{j\in J}$ is a basis for the isotropic subspace of $\H1(S;\mathbb{Z})$,
	\item \( i(\alpha_j, \alpha_k) = i(\beta_j, \beta_k) = 0 \) for all \(j,k \in I \),
	\item $i(\alpha_j,\beta_k)=\delta_{jk}$ for all $j,k\in I$, and
	\item for any compact subset $K$ of $S$, only finitely many curves in the basis intersect $K$.
\end{itemize}

\begin{lemma}
	For any surface $S$ there exists a geometric homology basis.
\end{lemma}

\begin{proof}
	Consider an exhaustion of $S$ by compact subsurfaces $\Sigma_1\subset \Sigma_2\subset\dots$ whose boundary curves are all separating and are allowed to be peripheral. Construct by induction a geometric symplectic basis for $\Sigma_n$ with boundary capped off, by choosing such a basis for $\Sigma_1$ and extending the basis for $\Sigma_n$ to a basis for $\Sigma_{n+1}$. This gives the $\{\alpha_i,\beta_i\}_{i\in I}$ with the required intersection numbers. To get the basis for the isotropic part, consider the set $\{\delta_l\}_{l\in L}$ of all boundary components of all $\Sigma_n$ and let $\{\gamma_j\}_{j\in J}$ be a maximal independent subset. The space generated by $\{\gamma_j\}_{j\in J}$ is the same as the space generated by $\{\delta_l\}_{l\in L}$ and this is the isotropic subspace of $\H1(S;\mathbb{Z})$. Furthermore, by construction, all $\gamma_j$ are pairwise disjoint and disjoint from the other curves. As any compact set is contained in some $\Sigma_n$, for $n$ big enough, and each subsurface contains only finitely many curves representing basis elements, we get the third property of a geometric basis as well.
\end{proof}

In the case of the Loch Ness monster surface, there is no isotropic subspace of $\H1(S;\mathbb{Z})$ of homology and a geometric homology basis is described in Figure \ref{fig:basisLochNess}.

\begin{figure}[H]
	\begin{overpic}[width=.5\textwidth]{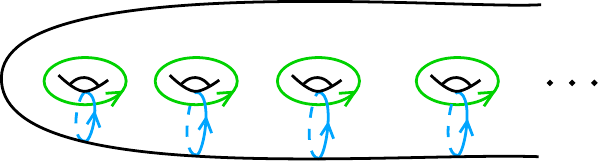}
		\put(14,19){$\alpha_1$}
		\put(33,19){$\alpha_2$}
		\put(52,19){$\alpha_3$}
		\put(75,19){$\alpha_4$}
		\put(12,-3){$\beta_1$}
		\put(31,-5){$\beta_2$}
		\put(52,-5){$\beta_3$}
		\put(76,-5){$\beta_4$}
	\end{overpic}
	\vspace*{.3cm}
	\caption{A geometric basis for the Loch Ness monster surface}
	\label{fig:basisLochNess}
\end{figure}
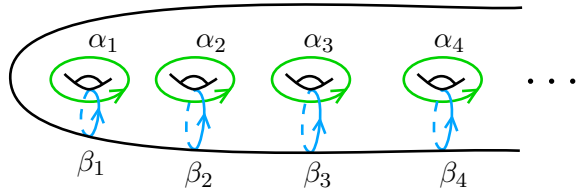
Note that the same homology basis could be realized by two different sets of curves, one of which gives a geometric homology basis and the other of which does not. An example is given in Figure \ref{fig:twobasesLochNess}.
The curves represent the same basis as the curves in Figure \ref{fig:basisLochNess}, but they are not a geometric homology basis: there is a compact subsurface intersecting all of the $\alpha_i'$.

\begin{figure}[H]
	\begin{overpic}[width=.55\textwidth]{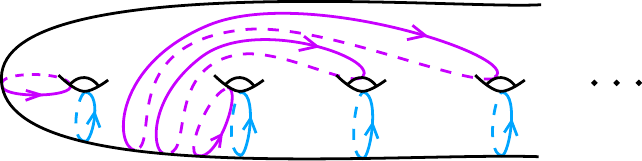}
		\put(5,15){$\alpha_1'$}
		\put(27,-5){$\alpha_2'$}
		\put(46,11){$\alpha_3'$}
		\put(72,19){$\alpha_4'$}
		\put(12,-3){$\beta_1$}
		\put(36,-5){$\beta_2$}
		\put(55,-5){$\beta_3$}
		\put(77,-5){$\beta_4$}
	\end{overpic}
		\vspace*{.3cm}
	\caption{Curves representing the same basis as the curves in Figure \ref{fig:basisLochNess}}\label{fig:twobasesLochNess}
\end{figure}
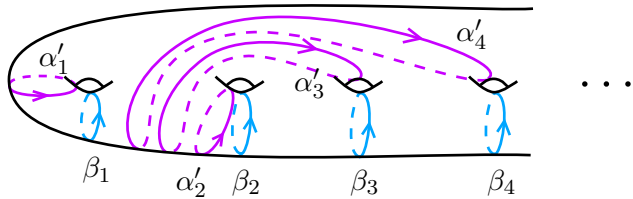

Moreover, two geometric symplectic bases do not need to be in the same mapping class group orbit, as the example in Figure \ref{fig:differentstdbases} shows.  
\begin{figure}[H]
	\includegraphics[width=.55\textwidth]{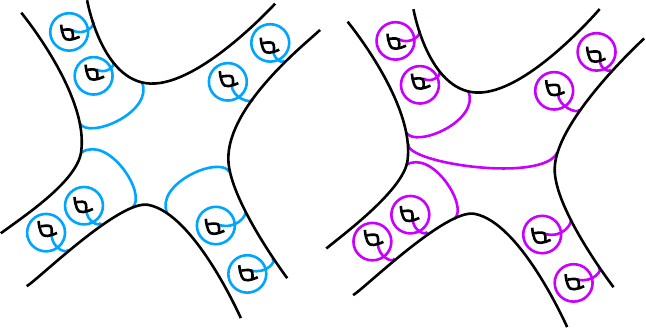}
	\caption{Two geometric symplectic bases that are not in the same mapping class group orbit.}\label{fig:differentstdbases}
\end{figure}

We will also consider the first cohomology group with integer coefficients $\C1(S;\Z)$ of a surface $S$.
We will often use the identification
$$\C1(S;\Z)\simeq \mbox{Hom}(\H1(S;\Z);\Z)$$
without explicit mention.

\subsection{Non-separating curves}\label{subsec:nonsep}
We begin with a simple criterion to detect classes realizable by non-separating simple closed curves,
which is likely well known.
\begin{lemma}[Lemma \ref{lem:detectingnonsep1}]
	Let $S$ be any surface and $x \in \H1(S;\Z)$.
	Then $x$ is a simple non-isotropic if and only if there exists $y \in \H1(S; \Z)$ such that $\ai(x,y) = 1$.
\end{lemma}
\begin{proof}
	First observe that for any non-separating simple closed curve $\alpha$ there exists a curve $\beta$ that intersects it geometrically once, and hence $\ai([\alpha],\pm[\beta])=1$. This shows one direction of the
	lemma.
	
	\smallskip For the other direction, we begin with the case where $S$ is finite type. Since $\ai(i_*x, i_*y) = \ai(x,y) = 1$, the homology
	class $i_*x \in \H1(\hat{S};\Z)$ is primitive. By Theorem~\ref{thm:mp}, this implies that $x$ is realized
	by a non-separating simple closed curve.
	
	In the case of a general $S$, there is a finite-type subsurface $F\subset S$ so that $x,y$ can be 
	realized by loops on $F$. Applying the previous case to $F$ then shows that $x$ can be realized
	by a simple closed curve on $F$, hence $S$.
\end{proof}
\begin{lemma}\label{lem:alphaprim}
	Suppose that $\{[\alpha_i], [\beta_i]\}_{i\in I}\cup\{[\gamma_j]\}_{j\in J}$ is a geometric basis for homology, where $\{[\gamma_j]\}_{j\in J}$ is a basis for the isotropic subspace of $\H1(S;\mathbb{Z})$.
	Then any primitive element in the span of $\{[\alpha_i], [\beta_i]\}_{i\in I}$ can
	be realized by a non-separating simple closed curve.
\end{lemma}
\begin{proof}
	Arguing as in the proof of Lemma~\ref{lem:detectingnonsep1}, by passing to a suitable subsurface it suffices to show this for surfaces of finite type.
	Now, the lemma follows from Theorem~\ref{thm:mp}, since $i_*$ induces an isomorphism between the span of the
	$\{[\alpha_i], [\beta_i]\}_{i\in I}$ and $\H1(\hat{S};\Z)$.
\end{proof}

Instead of requiring the existence of a class that intersects correctly, we can also characterize
classes realized by non-separating simple closed curves by an extension property.
\begin{lemma}\label{lem:detectingnonsep2}
	Let $S$ be any surface and $x \in \H1(S;\Z)$.
	Then $x$ is a simple non-isotropic if and only if $x$ is not isotropic
	and there is a basis $B$ of $\H1(S;\Z)$ so that 
	\begin{enumerate}[i)]
		\item $x\in B$, and
		\item $B$ contains a basis for the isotropic subspace of $\H1(S;\Z)$.
	\end{enumerate}
\end{lemma}
\begin{proof}
	One direction follows since any non-separating simple closed curve can be extended to a standard basis for homology.
	
	\smallskip For the reverse direction, suppose $x$ is not isotropic and can be extended to a basis $\{x, x_k, y_l\}$ where $\{y_l\}$ is a basis for the isotropic subspace.
	
	Choose some geometric basis for homology $\{[\alpha_i], [\beta_i]\}_{i\in I}\cup\{[\gamma_j]\}_{j\in J}$, where $\{[\gamma_j]\}_{j\in J}$ is a basis for the isotropic subspace of $\H1(S;\mathbb{Z})$.
	Write $x$ in this geometric basis as the sum
	\[ x =  X + Y \]
	where $X$ is a linear combination of the $\{[\alpha_i],[\beta_i]\}$ and $Y$ is a linear combination of the $\{[\gamma_j]\}$.
	
	Since $x$ is not isotropic we have $X\neq 0$. We claim that $X$ is in fact primitive. 
	Assuming this claim, by combining Lemma~\ref{lem:alphaprim} and Lemma~\ref{lem:detectingnonsep1}, we can find a homology class $w$ so that $1 = \ai(w,X)$. We then have $1 = \ai(w,X) = \ai(w,x)$, and we are done by Lemma~\ref{lem:detectingnonsep1}.
	
	\smallskip To prove the claim, suppose $X = nX_0$ for some $n\in \Z$.	Since $Y$ is isotropic we have
	\[ x-nX_0 = Y = \sum n_j y_j. \]
	Let $m\in \Z$ be the coefficient of $x$ when writing $X_0$ in the basis $\{x, x_k, y_l\}$.
	Then by looking at the coefficient of $x$ in the previous equation, we obtain
	$$1-nm=0$$
	which implies that $nm=1$, i.e.\ $n=\pm 1$.
\end{proof}
\begin{rmk}

In Lemma~\ref{lem:detectingnonsep2} is is actually necessary to require
that the basis $B$ contains a basis for the isotropic subspace.
Namely, let $S$ be a twice-punctured torus with standard geometric basis given by two non-separating curves $\alpha$ and $\beta$ intersecting once and one of the two boundary curves denoted by $\gamma$. Consider the class $x=[\gamma]+2[\alpha]$. It cannot be realized by a non-separating curve by Theorem \ref{thm:mp}, but it is non-isotropic and $[\alpha],[\beta],x$ is a basis of homology.
\end{rmk}

\subsection{Separating Curves}\label{subsec:sep}
To characterize simple isotropics, we will use intersections with arcs. Recall that we assume arcs to be properly embedded, but allow them to be non-compact.
\begin{lemma}\label{lem:detectingsep}
	Let $S$ be any surface and $x \in \H1(S;\Z)$. Then $x$ is a simple isotropic if and only if $x$ is isotropic and $|\ai(x, a)| \leq 1$ for every simple arc.
\end{lemma}
\begin{proof}
	One direction is easy: if $\gamma$ is a simple separating curve, then $\ai(\gamma, a)$ is $\pm 1$ if $\gamma$
	separates the ends that $a$ joins; otherwise the intersection is zero as any two successive intersections must have different signs.
	
	\smallskip We begin by showing the other direction in the case of a finite-type surface $S$ with $n+1$ punctures. 
	Let $\gamma_1, \ldots,  \gamma_{n+1}$ be loops, each surrounding a puncture and oriented so that the puncture is to the left.
	Then any collection of $n$ elements of the set $\{ \gamma_1, \ldots, \gamma_{n+1} \}$ is a basis for the isotropic subspace of $\H1(S, \Z)$. 
	Let $x$ be an isotropic homology class and suppose that $|\ai(x, a)| \leq 1$ for every arc joining punctures.
	As \( x \) is isotropic, we have
	\[ x = \sum_{i=1}^n c_i [\gamma_i]. \]
	Consider an arc $\alpha$ from $\gamma_j$ to $\gamma_{n+1}$; then
	$$1\geq|\ai(x,\alpha)|=\left|\sum_{i=1}^n c_i \ai([\gamma_i],\alpha)\right|=|c_j|$$
	so each coefficient has absolute value at most one.
	If there were two indices such that $c_j=1$ and $c_k=-1$, then any arc connecting $\gamma_j$ and $\gamma_k$ would have algebraic intersection number $\pm 2$ with $x$.
	Hence all non-zero coefficients have the same sign and using Theorem~\ref{thm:mp} we deduce that $x$ can be realized by a simple closed curve.
	
	\smallskip Finally, suppose that $S$ is of infinite type, and that $x$ is as in the lemma. Choose
	a finite type surface $F \subset S$ which contains a loop homologous to $x$ and so that every boundary
	curve of $F$ is separating in $S$. By the latter property, any simple arc $a_0 \subset F$ can be extended to a
	simple arc $a$ in $S$ so that $a \cap F = a_0$. Hence, we can apply the finite-type case to $x$ and $F$, and
	conclude that $x$ is a simple isotropic in $F$, and hence in $S$.
\end{proof}

\subsection{Arcs}\label{subsec:arcs}
Given an arc $\alpha$ joining two ends, we have an associated integral cohomology class $\ai(\alpha,\cdot)$. The goal of this section is to characterize which cohomology classes arise this way.

Given $f\in\C1(S,\Z)$, we say that:
\begin{itemize}
\item $f$ \emph{has support in $e\in\Ends(S)$} if for all $V\in \filt_e$ there is $x\in V$ such that $f(x)\neq 0$;
\item the \emph{support} $\supp(f)$ of $f$ is the set of ends in which $f$ has support;
\item $f$ is \emph{arclike} in $e\in\Ends(S)$ if for all $V\in \filt_e$ there is $x\in V$ such that $x$ isotropic and $f(x)=1$.
\end{itemize} 

The goal of the next set of lemmas is to prove the following:
\begin{prop}\label{prop:detectingarcs}
Let $f\in\C1(S;\Z)$ be such that
\begin{enumerate}
\item the support of $f$ is $\{e_1,e_2\}$, for $e_1\neq e_2\in\Ends(S)$, and
\item $f$ is arclike in $e_1$ and $e_2$.
\end{enumerate}
Then $f=\ai(\alpha,\cdot)$ for a simple arc $\alpha$ connecting $e_1$ and $e_2$.
\end{prop}

We start with a lemma concerning the support of cohomology classes:

\begin{lemma}\label{lem:supp&star}
Let $f\in\C1(S;\Z)$. Suppose $\Sigma$ is a star surface with a boundary component $\gamma$ such that $f([\gamma])\neq 0$.
Then $\supp(f)\cap\rends(\gamma)\neq \emptyset$.
\end{lemma}

\begin{proof}
Fix an exhaustion by star surfaces $F_n$.

Denote by $X_{\gamma}$ the component of $S\ssm\gamma$ to the right of $\gamma$. 
Choose \( N \) so that \( \gamma \) is contained in the interior of \( F_N \).
Then $\Sigma_1:=F_N\cap X_{\gamma}$ is a star surface with $\Sigma\cap\Sigma_1=\gamma$. 
The sum of the classes of the boundary components and of the curves surrounding one puncture of $\Sigma_1$ is zero in homology. So there must be one such curve $\gamma_1$, different from $\gamma$, which satisfies $f([\gamma_1])\neq 0$.

If $\rends(\gamma_1)$ is a single puncture, we are done. Otherwise we can repeat the process with $\Sigma_1$ and $\gamma_1$ instead of $\Sigma$ and $\gamma$.

If we find a $\Sigma_n$ and a curve $\gamma_n\subset \Sigma_n$ with $\rends(\gamma_n)=\{e\}$, we conclude as above that $e\in\supp(f)\cap\rends(\gamma_n)\subset\supp(f)\cap\rends(\gamma)$.

Otherwise, we get an infinite sequence of surfaces and a sequence curves $\gamma_n$ going to infinity, and hence accumulating to an end $e$ in $\rends(\gamma)$, on which $f$ is non-zero.
Thus $e\in\supp(f)$.
\end{proof}

An easy consequence of the previous lemma is the following:
\begin{cor}\label{cor:max2boundaries}
If $f\in\C1(S;\Z)$ has $\supp(f)=\{e_1,e_2\}$, then for any star surface $\Sigma$ there are at most two boundary curves $\gamma_1$ and $\gamma_2$ on which $f$ is  non-zero. \qed
\end{cor}

We will also need a characterization of intersection with arcs in the case of finite-type surfaces.
It relies on the following lemma.
\begin{lemma}\label{lem:realise-scc}
Suppose that $\Sigma$ is a closed surface of finite type and that $\alpha$ is a simple closed curve.
If $f$ is a cohomology class on $\Sigma$ with $f([\alpha]) = 1$, then there is a simple closed curve $\beta$ so that 

\begin{enumerate}
\item $f(x) = \ai(x, [\beta])$ for all $x\in \H1(\Sigma;\Z)$, and
\item $\alpha$ and $\beta$ intersect in a single point.
\end{enumerate}
\end{lemma}
\begin{proof}
The algebraic intersection form $\ai$ is a non-degenerate symplectic form in this case, so there exists a homology class $b$ with $f(x) = \ai(x, b)$ for all $x\in \H1(\Sigma;\Z)$.
Since   $f(\alpha)=1$, the class $b$ is primitive, and can thus (by Theorem \ref{thm:mp}) be realized by a simple closed curve $\beta$, showing (1). 
The fact that $\beta$ can be chosen to intersect $\alpha$ in a single point can be shown as in \cite[Theorem 6.4]{FM_Primer}.
\end{proof}

We can now prove the finite-type version of Proposition \ref{prop:detectingarcs}.

\begin{lemma}\label{lem:finite-arcs}
Let $\Sigma$ be a compact surface and \( \gamma_1, \ldots, \gamma_n \) be its boundary components. Suppose $f\in\C1(\Sigma;\Z)$ is such that there are two indices $i_1, i_2$ for which
\[ f([\gamma_{i_1}]) = -f([\gamma_{i_2}]) = 1,\]
 and
\[ f([\gamma_j]) = 0, \quad\forall j \neq i_1, i_2. \]

Then $f = \ai(\alpha,\cdot)$ for a simple arc $\alpha$ connecting $\gamma_{i_1}$ and $\gamma_{i_2}$.
\end{lemma}

\begin{proof}
  Let $F$ be the surface obtained from $\Sigma$ by gluing $\gamma_{i_1}$ and $\gamma_{i_2}$. We can describe the homology of $F$ as follows:
  
  \[ \H1(F;\Z) = \Z\oplus \H1(\Sigma;\Z) / \langle [\gamma_{i_1}] +
    [\gamma_{i_2}] \rangle, \] where the obvious map from $\Sigma$ to
  $F$ corresponds to the quotient map
  $$\H1(\Sigma;\Z) \to \H1(\Sigma;\Z) / \langle [\gamma_{i_1}] +
  [\gamma_{i_2}] \rangle.$$

  By the first assumption, the cohomology class $f$ on $\Sigma$ descends to a
  form on
 $\H1(\Sigma;\Z) / \langle [\gamma_{i_1}] +
  [\gamma_{i_2}] \rangle$, and we extend this to a form $f_F$ on
  $\H1(F;\Z)$ by letting it be $0$ on the summand $\Z$.

  Let $S$ be the closed surface obtained by gluing a disc $D_j$
  to the boundary component $\gamma_j$ of $F$ for each
  $j\neq i_1,i_2$.   Observe that
  \[ \H1(S;\Z) = \H1(F;\Z) / \langle [\gamma_j], j\neq i_1,i_2
    \rangle \] 
   By the second assumption on $f$, the class $f_F$
  descends to a cohomology class $f_S$ on $S$. Observe that
  $f_S(\gamma_{i_1}) = 1$, and hence we can apply
  Lemma~\ref{lem:realise-scc} to find a curve $\bar{\beta}$
  which intersects $\gamma_{i_1}$ in a single point, and which satisfies
  \[ \iota(x,\bar{\beta}) = f_S(x), \quad\forall x \in
    \H1(S;\Z). \]
 We may assume that $\bar{\beta}$
  is disjoint from all the discs $D_j$, and therefore defines a curve
  $\beta\subset F$, which now has
  \[ \iota(x,\beta) =f_F(x), \quad\forall x \in \H1(F;\Z). \]
  and still intersects $\gamma_{i_1}$ in a single point.

  The preimage of $\beta$ on $\Sigma$ is then the desired arc.
\end{proof}

\begin{proof}[Proof of Proposition \ref{prop:detectingarcs}]
In this proof, we will allow subsurfaces to have boundary components homotopic to punctures and  to be annuli with (both) boundary curves homotopic to a puncture.

Consider an isotropic class $x$ such that $f(x)=1$ and let $\Sigma_0$ be a compact star surface such that $x\in\H1(\Sigma_0,\Z)$.
Let $\gamma_0^1,\dots \gamma_0^m$ be the boundary components of $\Sigma_0$.

Choose a compact exhaustion $F_n$ such that $\Sigma_0=F_0$ and no two boundary components are homotopic, unless they are homotopic to a single puncture.

By Corollary \ref{cor:max2boundaries}, there are at most two boundary components of $\Sigma_0$, say $\gamma_0^1$ and $\gamma_0^2$, on which $f$ is not zero.
Since the sum of the boundary components of \( \Sigma_0 \)  is zero in homology, $f(\gamma_0^1)=-f(\gamma_0^2)$. 
Furthermore $\{[\gamma_0^i]\st i=1,\dots, m-1\}$ is a basis of the isotropic part, hence there are $c_i\in\Z$ with
$$x=\sum_{i=1}^{m-1}c_i[\gamma_0^i].$$
Applying $f$ we get
$$1=c_1 f(\gamma_0^1)+c_2 f(\gamma_0^2)=(c_1-c_2)f(\gamma_0^1)$$
which implies that $f(\gamma_0^1)=\pm 1$ and $f(\gamma_0^2)=\mp 1$.
Note that by Lemma \ref{lem:supp&star} (and up to changing the labels of the end in the support of $f$), we have $e_i\in \lends(\gamma_0^i)$.
As $f$ has support only in two ends, up to enlarging $\Sigma_0$ we can assume that $f\neq 0$ only on $\Sigma_0$ union the two connected components of $S\smallsetminus \Sigma_0$ which have $\gamma_0^1$ or $\gamma_0^2$ in their boundary, i.e.\ those containing $e_1$ or $e_2$.
Now apply Lemma \ref{lem:finite-arcs} to get that $f\vert_{\Sigma_0}=\ai(\alpha_0,\cdot)$ for some simple arc $\alpha_0\subset \Sigma_0$ connecting $\gamma_0^1$ to $\gamma_0^2$.
Denote by $X_0^i$ the flare surface to the right of $\gamma_0^i$.

\begin{figure}[h]
\begin{overpic}{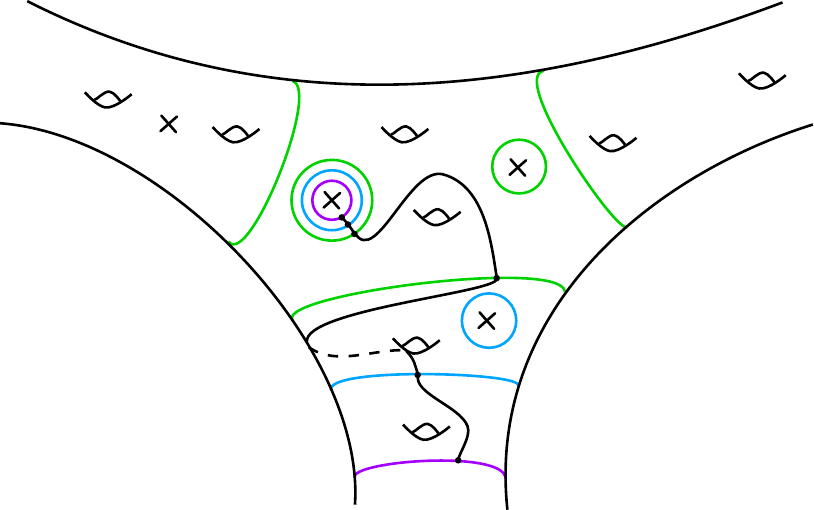}
\put(45,55){$\Sigma_0$}
\put(35,28){$\gamma_1^1$}
\put(68,18){$\Sigma_1^1$}
\put(33,13){$\gamma_2^1$}
\put(65, 8){$\Sigma_2^1$}
\put(38,2){$\gamma_3^1$}
\end{overpic}
\caption{An example of the sequence of subsurfaces constructed in the proof of Proposition \ref{prop:detectingarcs}}
\end{figure}
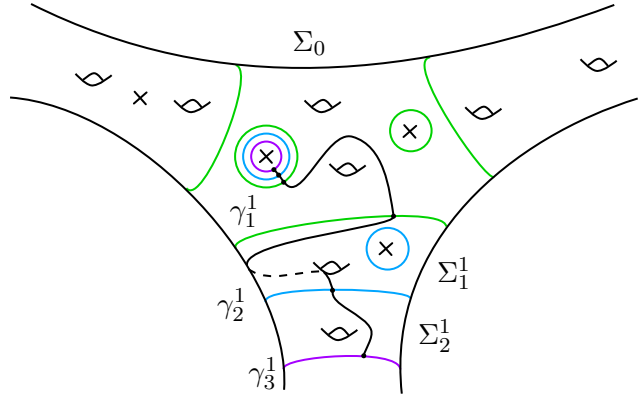

Consider $\Sigma_1^i:=F_1\cap X_0^i$ for $i=1,2$ and if necessary enlarge them so that $f$ is zero on all components of $S\smallsetminus (\Sigma_0\cup\Sigma_1^1\cup\Sigma_1^2)$ not containing $e_1$ or $e_2$.
Repeat the argument to get that $f$ is given by  intersection in arcs $\alpha_1^i$ in these subsurfaces.
Slide the endpoints on the boundary components so that the arcs can be glued to $\alpha_0$; the resulting arc defines $f$ on the union of the three star surfaces.

Repeat the process to get that $f$ is represented by $\ai(\alpha,\cdot)$ on a (countably infinite) union of star surfaces, where $\alpha$ is a simple arc joining $e_1$ and $e_2$. By construction, $f$ is zero outside of the union, hence $f=\ai(\alpha,\cdot)\in\C1(S;\Z)$. 
\end{proof}

\bibliographystyle{alpha}
\bibliography{references_MCGhomology}

\newcommand{\etalchar}[1]{$^{#1}$}
\begin{thebibliography}{AGK{\etalchar{+}}18}

\bibitem[AGK{\etalchar{+}}18]{AramayonaBig}
Javier Aramayona, Tyrone Ghaswala, Autumn~E Kent, Alan McLeay, Jing Tao, and
  Rebecca~R Winarski.
\newblock Big torelli groups: generation and commensuration.
\newblock {\em arXiv preprint arXiv:1810.03453}, 2018.

\bibitem[APV17]{APV_Cohomology}
Javier Aramayona, Priyam Patel, and Nicholas~G. Vlamis.
\newblock The first integral cohomology of pure mapping class groups.
\newblock {\em arXiv preprint arXiv:1711.03132}, 2017.

\bibitem[Bur89]{Burkhardt_Grundzuege}
Heinrich Burkhardt.
\newblock Grundz\"{u}ge einer allgemeinen {S}ystematik der hyperelliptischen
  {F}unctionen {I}. {O}rdnung.
\newblock {\em Math. Ann.}, 35(1-2):198--296, 1889.

\bibitem[FM12]{FM_Primer}
Benson Farb and Dan Margalit.
\newblock {\em A primer on mapping class groups}, volume~49 of {\em Princeton
  Mathematical Series}.
\newblock Princeton University Press, Princeton, NJ, 2012.

\bibitem[Joh83]{JohnsonSurvey}
Dennis Johnson.
\newblock A survey of the {T}orelli group.
\newblock In {\em Low-dimensional topology ({S}an {F}rancisco, {C}alif.,
  1981)}, volume~20 of {\em Contemp. Math.}, pages 165--179. Amer. Math. Soc.,
  Providence, RI, 1983.

\bibitem[MP78]{mp_representing}
William~H. Meeks, III and Julie Patrusky.
\newblock Representing homology classes by embedded circles on a compact
  surface.
\newblock {\em Illinois J. Math.}, 22(2):262--269, 1978.

\bibitem[PV18]{PV_Algebraic}
Priyam Patel and Nicholas~G. Vlamis.
\newblock Algebraic and topological properties of big mapping class groups.
\newblock {\em Algebr. Geom. Topol.}, 18(7):4109--4142, 2018.

\bibitem[Ric63]{Richards_Classification}
Ian Richards.
\newblock On the classification of noncompact surfaces.
\newblock {\em Trans. Amer. Math. Soc.}, 106:259--269, 1963.

\end{thebibliography}

\end{document}